\documentclass[12pt]{amsart}
\usepackage{amsmath}
\usepackage{amssymb}
\usepackage{graphicx}
\usepackage[margin = .9in]{geometry}
\usepackage{MnSymbol}
\usepackage{enumerate}
\usepackage{slantsc}
\usepackage{tikz}
\usepackage{mathtools}
\usepackage{calc}
\usepackage{amsthm}
\usepackage{cleveref}
\usepackage{latexsym}
\usepackage{booktabs}
\usepackage{array}
\usepackage{multicol}
\usepackage{emp}
\usepackage[shellescape,latex]{gmp}
\usepackage{leftidx}
\date{\today}
\title{Integrability of Free Noncommutative Functions}
\author{Dmitry Kaliuzhnyi-Verbovetskyi}
\address{Department of Mathematics \\
Drexel University\\
3141 Chestnut St.\\
Philadelphia, PA, 19104}
\email{dmitryk@math.drexel.edu}
\author{Leonard Stevenson}
\address{Department of Mathematics \\
Drexel University\\
3141 Chestnut St.\\
Philadelphia, PA, 19104}
\email{ls879@drexel.edu}
\author{Victor Vinnikov}
\address{Department of Mathematics\\
Ben-Gurion University of the Negev\\
Beer-Sheva, Israel, 84105} 
\email{vinnikov@math.bgu.ac.il}

\thanks{DK-V was partially supported by NSF grant DMS-0901628.
DK-V and VV were partially supported by BSF grant 2010432, and by
ANR-11-LABX-0040-CIMI within the program ANR-11-IDEX-0002-02. They also wish to thank Serban Belinschi for organizing the workshop ``Complex Analysis and Noncommutative Functions", 3--7 October 2016, in CIMI, Toulouse, where a significant progress in this research has been achieved.}

\renewcommand{\(}{\left(}
\renewcommand{\)}{\right)}
\renewcommand{\[}{\left[}
\renewcommand{\]}{\right]}
\let\oldsum\sum
\renewcommand{\sum}{\displaystyle\oldsum}

\newlength{\whatsleft}

\newcounter{ProblemNum}
\newcounter{SubProblemNum}[ProblemNum]

\renewcommand{\theSubProblemNum}{\alph{SubProblemNum}}

\renewcommand*{\part}{\stepcounter{SubProblemNum} %
  {\bf (\theSubProblemNum)}\hspace{1mm}}

\newcommand{\row}{\rm{row}}

\newcommand{\nc}{\rm{nc}}
\newcommand{\cb}{\rm{cb}}

\newcommand{\set}[2]{\{ {#1} \hspace{2mm} | \hspace{2mm} {#2} \}}

\newcommand{\cm}[1]{&\qquad &\text{ {#1} }} 

\newcommand{\cmm}[1]{&\quad &\hspace{1.3mm}\begin{minipage}[c]{3in} \raggedright {#1} \end{minipage}}

\newcommand{\fmat}[1]{\[ \begin{array}{#1} } 
\newcommand{\emat}{\end{array} \]} 

\newcommand{\R}{\mathcal{R}}
\newcommand{\M}{\mathcal{M}}
\newcommand{\V}{\mathcal{V}}

\newcommand{\W}{\mathcal{W}}

\newcommand{\N}{\mathcal{N}}
\newcommand{\NN}{\mathbb{N}}
\newcommand{\U}{\mathcal{U}}

\newcommand{\A}{\mathcal{A}}

\newcommand{\T}{\mathcal{T}}

\renewcommand{\L}{\mathcal{L}}

\numberwithin{equation}{section}

\theoremstyle{plain}

\newtheorem{theorem}{Theorem}[section]
\newtheorem{corollary}[theorem]{Corollary}
\newtheorem{proposition}[theorem]{Proposition}
\newtheorem{lemma}[theorem]{Lemma}

\newenvironment{remark}[1][Remark:]{\begin{trivlist}
\item[\hskip \labelsep {\bfseries #1}]}{\end{trivlist}}
\newenvironment{definition}[1][Definition]{\begin{trivlist}
\item[\hskip \labelsep {\bfseries #1}]}{\end{trivlist}}

\DeclareMathSymbol\unlhd{\mathrel}{lasy}{"02}

\makeatletter
\newcommand{\thickhline}{%
    \noalign {\ifnum 0=`}\fi \hrule height 1pt
    \futurelet \reserved@a \@xhline
}
\newcolumntype{"}{@{\hskip\tabcolsep\vrule width 1pt\hskip\tabcolsep}}
\makeatother

\begin{document}

\begin{abstract}
In \cite{Verb}, an algebraic construction is used to develop a difference-differential calculus for free noncommutative functions. This paper gives necessary and sufficient conditions for higher order free noncommutative functions to have an antiderivative. 
\end{abstract}

\maketitle

\section*{Introduction}
Noncommutative functions are graded functions between sets of square matrices of all sizes over two vector spaces
that respect direct sums and similarities.
They possess very strong regularity properties (reminiscent
of the regularity properties of usual analytic functions) and admit a good difference-differential calculus.
Noncommutative functions appear naturally in a large variety of settings: noncommutative algebra,
systems and control, spectral theory, and free probability. Starting with pioneering work of J. L. Taylor \cite{Tay72b,Tay73}, the theory was further developed by D.-V. Voiculescu \cite{Voi04,Voi09},
and established itself in recent years as a new and extremely active
research area -- see, e.g., \cite{HKMcC1,HKMcC2,Po06,Po10,MS,AgMcC6}, and the monograph of two of the authors \cite{Verb}. 

The goal of the present paper is to establish a noncommutative analog of the Frobenius integrability theorem.

We begin by providing the necessary background on noncommutative (nc) sets and free nc  functions; we refer the reader to \cite{Verb} for more detail. Let $\R$ be a unital, commutative ring and $\M, \N$ be $\R$-modules. Denote by $\M_{\nc}$ the set of all square matrices of all sizes with entries from $\M$. Then a subset $\Omega$ of $\M_{\nc}$ is called a nc set if it is closed under direct sums, i.e.,
$$X,Y\in\Omega\Longrightarrow X\oplus Y=\begin{bmatrix}
X & 0\\
0 & Y
\end{bmatrix}\in\Omega.$$
 Define 
\begin{enumerate}
\item $\Omega_n$ to be the set of all $n \times n$ matrices in $\Omega$, and

\item $\Omega_{\rm d.s.e.} = \set{X \in \M_{\nc}}{X \oplus Y \in \Omega \text{ for some } Y \in \M_{\nc}}$.
\end{enumerate}
Further, $\Omega$ is called right admissible if for all $X \in \Omega_n, Y \in \Omega_m$ and for all matrices $Z \in \M^{n \times m}$, there exists an invertible $r \in \R$ such that $\fmat{cc} X & rZ \\ 0 & Y \emat \in \Omega_{n + m}$. 

A mapping $f\colon \Omega \to \N_{\nc}$ satisfying $f(\Omega_n) \subseteq \N^{n \times n}$ for all $n \in \NN$ is called a (free) nc function if it respects direct sums and similarities, i.e.,
\begin{alignat*}{2}
f(X \oplus Y) &= f(X) \oplus f(Y) \cm{for all $X \in \Omega_n, Y \in \Omega_m$,} \\
f(SXS^{-1}) &= Sf(X)S^{-1} \cmm{for all $X \in \Omega_n$ and invertible $S \in \R^{n \times n}$ such that $SXS^{-1} \in \Omega$.} \end{alignat*}
On nc sets, this pair of conditions is equivalent to the condition that $f$ respects intertwinings. That is, given $X \in \Omega_n, Y \in \Omega_m$, and $S \in {\R}^{m \times n}$,
$$SX = YS \implies Sf(X) = f(Y)S.$$

When $\Omega$ is a right admissible nc set, a difference-differential operator $\Delta = \Delta_R $ acting on $f$ can be defined by evaluating $f$ on block upper triangular matrices,

$$f\(\fmat{cc} X & Z \\ 0 & Y \emat\) = \fmat{cc} f(X) & \Delta f(X, Y)(Z) \\ 0 & f(Y) \emat.$$
The new function, $\Delta f$, can be extended to a linear function of $Z$ and is shown to have the following properties with respect to direct sums and similarities,
\begin{alignat*}{2}
\Delta f(X^0 \oplus X^1, Y)\(\fmat{c} Z_1 \\ Z_2 \emat\) &= \fmat{c} \Delta f(X^0, Y)(Z_1) \\ \Delta f(X^1, Y)(Z_2) \emat, \\
\Delta f(X, Y^0 \oplus Y^1)(\fmat{cc} Z_1 & Z_2 \emat) &= \fmat{cc} \Delta f(X, Y^0)(Z_1) & \Delta f(X, Y^1)(Z_2) \emat, \\
\Delta f(SXS^{-1}, Y)(SZ) &= S\Delta f(X, Y)(Z), \\
\Delta f(X, SYS^{-1})(ZS^{-1}) &= \Delta f(X, Y)(Z)S^{-1}. \end{alignat*}
Equivalently, $\Delta f$ is said to respect intertwinings,
\begin{alignat*}{2}
SX = WS \implies S\Delta f(X, Y)(Z) &= \Delta f(W, Y)(SZ), \\
SY = WS \implies \Delta f(X, Y)(ZS) &= \Delta f(X, W)(Z)S. \end{alignat*} 

More generally, let $\M_0, \ldots, \M_k, \N_0, \ldots, \N_k$ be $\R$-modules and let $\Omega^{(0)}, \ldots, \Omega^{(k)}$ be nc sets in $\M_{0, \nc}, \ldots, \M_{k, \nc}$, respectively. Then, a function $f$ defined on $\Omega^{(0)} \times \hdots \times \Omega^{(k)}$ with values $k$-linear mappings from
$\N_1^{n_0 \times n_1} \times \hdots \times \N_k^{n_{k - 1}\times n_k}$ to $\N_0^{n_0 \times n_k}$, or equivalently,
$$f(X^0, \ldots, X^k) \in \hom_{\R}(\N_1^{n_0 \times n_1} \otimes \hdots \otimes \N_k^{n_{k - 1} \times n_k}, \N_0^{n_0 \times n_k})$$
for all $(X^0, \ldots, X^k) \in \Omega^{(0)}_{n_0} \times \hdots \times \Omega^{(k)}_{n_k}$ and all $n_0, \ldots, n_k \in \NN$, is called an order $k$ (free) nc function if, given matrices of appropriate sizes, it respects direct sums and similarities, i.e.,
\begin{alignat*}{2}
&\begin{aligned}
&f(X^0_1 \oplus X^0_2, X^1, \ldots, X^k)\(\fmat{c} Z^1_1 \\ Z^1_2 \emat, Z^2, \ldots, Z^k\) = \fmat{c} f(X^0_1, X^1, \ldots, X^k)(Z^1_1, Z^2, \ldots, Z^k) \\ f(X^0_2, X^1, \ldots, X^k)(Z^1_2, Z^2, \ldots, Z^k) \emat, \end{aligned} \\
&\begin{aligned}
&f(X^0, \ldots, X^{j - 1}, X^j_1 \oplus X^j_2, X^{j + 1}, \ldots, X^k)\(Z^1, \ldots, Z^{j - 1}, \fmat{cc} Z^j_1 & Z^j_2 \emat, \fmat{cc} Z^{j + 1}_1 \\ Z^{j + 1}_2 \emat, Z^{j + 2}, \ldots, Z^k\) \\
&\hspace{15mm}= f(X^0, \ldots, X^{j - 1}, X^j_1, X^{j + 1}, \ldots, X^k)(Z^1, \ldots, Z^{j - 1}, Z^j_1, Z^{j + 1}_1, Z^{j + 2}, \ldots, Z^k) \\
&\hspace{15mm}+ f(X^0, \ldots, X^{j - 1}, X^j_2, X^{j + 1}, \ldots, X^k)(Z^1, \ldots, Z^{j - 1}, Z^j_2, Z^{j + 1}_2, Z^{j + 2}, \ldots, Z^k), \end{aligned} \\
&\begin{aligned}
&f(X^0, \ldots, X^{k - 1}, X^k_1 \oplus X^k_2)(Z^1, \ldots, Z^{k - 1}, \fmat{cc} Z^k_1 & Z^k_2 \emat) \\
&\hspace{15mm}= \fmat{cc} f(X^0, \ldots, X^{k - 1}, X^k_1)(Z^1, \ldots, Z^{k - 1}, Z^k_1) & f(X^0, \ldots, X^{k - 1}, X^k_2)(Z^1, \ldots, Z^{k - 1}, Z^k_2) \emat, \end{aligned} \\
&\begin{aligned}
&f(S_0X^0S_0^{-1}, X^1, \ldots, X^k)(S_0Z^1, Z^2, \ldots, Z^k) = S_0f(X^0, \ldots, X^k)(Z^1, \ldots, Z^k), \end{aligned} \\
&\begin{aligned}
&f(X^0, \ldots, X^{j - 1}, S_jX^jS_j^{-1}, X^{j + 1}, \ldots, X^k)(Z^1, \ldots, Z^{j - 1}, Z^jS_j^{-1}, S_jZ^{j + 1}, Z^{j + 2}, \ldots, Z^k) \\
&\hspace{15mm}= f(X^0, \ldots, X^k)(Z^1, \ldots, Z^k), \end{aligned} \\
&\begin{aligned}
&f(X^0, \ldots, X^{k - 1}, S_kX^kS_k^{-1})(Z^1, \ldots, Z^{k - 1}, Z^kS_k^{-1}) = f(X^0, \ldots, X^k)(Z^1, \ldots, Z^k)S_k^{-1}. \end{aligned} \end{alignat*}
Equivalently, for appropriately sized matrices, it respects intertwinings as follows:
\begin{alignat*}{2}
&T_0X^0_1 = X^0_2T_0 \implies T_0f(X^0_1, X^1, \ldots, X^k)(Z^1, \ldots, Z^k) \\
&\hspace{40mm} = f(X^0_2, X^1, \ldots, X^k)(T_0Z^1, Z^2 \ldots, Z^k), \\
&T_jX^j_1 = X^j_2T_j \implies f(X^0, \ldots, X^{j - 1}, X^j_1, X^{j + 1}, \ldots, X^k)(Z^1, \ldots, Z^{j - 1}, Z^jT_j, Z^{j + 1}, \ldots, Z^k) \\
&\hspace{40mm}= f(X^0, \ldots, X^{j - 1}, X^j_2, X^{j + 1}, \ldots, X^k)(Z^1, \ldots, Z^j, T_jZ^{j + 1}, Z^{j + 2}, \ldots, Z^k), \\
&T_kX^k_1 = X^k_2T_k \implies f(X^0, \ldots, X^{k - 1}, X^k_1)(Z^1, \ldots, Z^{k - 1}, Z^kT_k) \\
&\hspace{40mm} = f(X^0, \ldots, X^{k - 1}, X^k_2)(Z^1, \ldots, Z^k)T_k. \end{alignat*}
In this case, we write $f\in\T^k(\Omega^{(0)},\ldots,\Omega^{(k)};\N_{0,\nc},\ldots,\N_{k,\nc})$.

Under this definition, we say our original nc functions are of order $0$.

 The difference-differential operator can be extended to order $k$ nc functions as follows:
\begin{alignat*}{2}
&f\(\fmat{cc} X^0_1 & Z \\ 0 & X^0_2 \emat, X^1, \ldots, X^k\)\(\fmat{c} Z^1_1 \\ Z^1_2 \emat, Z^2, \ldots, Z^k\) \\
&\hspace{10mm}= \fmat{c} f(X^0_1, X^1, \ldots, X^k)(Z^1_1, Z^2, \ldots, Z^k) + {_0\Delta} f(X^0_1, X^0_2, X^1, \ldots, X^k)(Z, Z^1_2, Z^2, \ldots, Z^k) \\ f(X^0_2, X^1, \ldots, X^k)(Z^1_2, Z^2, \ldots, Z^k) \emat, \\
&f\(X^0, \ldots, X^{j - 1}, \fmat{cc} X^j_1 & Z \\ 0 & X^j_2 \emat, X^{j + 1}, \ldots, X^k\)\(Z^1, \ldots, Z^{j - 1}, \fmat{cc} Z^j_1 & Z^j_2 \emat, \fmat{c} Z^{j + 1}_1 \\ Z^{j + 1}_2 \emat, Z^{j + 2}, \ldots, Z^k\) \\
&\hspace{15mm}= f(X^0, \ldots, X^{j - 1}, X^j_1, X^{j + 1}, \ldots, X^k)(Z^1, \ldots, Z^{j - 1}, Z^j_1, Z^{j + 1}_1, Z^{j + 2}, \ldots, Z^k) \\
&\hspace{15mm}+ {_j\Delta} f(X^0, \ldots, X^{j - 1}, X^j_1, X^j_2, X^{j + 1}, \ldots, X^k)(Z^1, \ldots, Z^{j - 1}, Z^j_1, Z, Z^{j + 1}_2, Z^{j + 2}, \ldots, Z^k) \\
&\hspace{15mm}+ f(X^0, \ldots, X^{j - 1}, X^j_2, X^{j + 1}, \ldots, X^k)(Z^1, \ldots, Z^{j - 1}, Z^j_2, Z^{j + 1}_2, Z^{j + 2}, \ldots, Z^k), \\
&f\(X^0, \ldots, X^{k - 1}, \fmat{cc} X^k_1 & Z \\ 0 & X^k_2 \emat\)\(Z^1, \ldots, Z^{k - 1}, \fmat{cc} Z^k_1 & Z^k_2 \emat\) \\ 
&\hspace{15mm}= {\row} \[f(X^0, \ldots, X^{k - 1}, X^k_1)(Z^1, \ldots, Z^{k - 1}, Z^k_1), \right. \\ 
&\hspace{55mm} \left. {_k\Delta} f(X^0, \ldots, X^{k - 1}, X^k_1, X^k_2)(Z^1, \ldots, Z^{k - 1}, Z^k_1, Z) \right. \\
&\hspace{90mm} + \left. f(X^0, \ldots, X^{k - 1}, X^k_2)(Z^1, \ldots, Z^{k - 1}, Z^k_2)\]. 
\end{alignat*}
In each case ${_j\Delta} f$, $j = 0, \ldots, k$, yields an nc function of order $k + 1$, so that
\begin{multline*}
{_j\Delta}\colon\T^k(\Omega^{(0)},\ldots,\Omega^{(k)};\N_{0,\nc},\ldots,\N_{k,\nc})\\
\to\T^k(\Omega^{(0)},\ldots,\Omega^{(j-1)},\Omega^{(j)},\Omega^{(j)},\Omega^{(j+1)},\ldots,\Omega^{(k)};\N_{0,\nc},\ldots,\N_{j,\nc},\M_{j,\nc},\N_{j+1,\nc},\ldots,\N_{k,\nc}).
\end{multline*}

This paper considers the process of undoing the operators ${_j\Delta} $. When $k = 0$, this means we are given a nc function, $F$, of order $1$. It is proved that there exists an nc function, $f$, of order $0$ such that $\Delta f = F$ if and only if ${_0\Delta} F = {_1\Delta} F$. This is done in the following steps.

First, an order $0$ nc function $f$ is defined up to the selection of the value of $f$ at some arbitrary  point $Y \in \Omega_s$. This definition is inspired by formula (2.19) in \cite{Verb}: given $X \in \Omega_{sm}$,
$$f(X) = I_m \otimes f(Y) + \Delta f(I_m \otimes Y, X)(X - (I_m \otimes Y)).$$
It is then shown that this will yield an nc function $f$ for which $\Delta f = F$ if and only if there exists a value $f(Y)$ such that 
\begin{equation} \label{innerderivationproperty}
Tf(Y) - f(Y)T = F(Y, Y)(TY - YT) \end{equation}
for all matrices $T \in \R^{s \times s}$.
Next, we define $D_Y\colon \R^{s \times s}\to \N^{s \times s}$ by $D_Y(S) = F(Y, Y)(SY - YS)$, and show that $D_Y$ is a Lie-algebra derivation:
\begin{alignat}{3}
D_Y(S)T - TD_Y(S) + SD_Y(T) - D_Y(T)S = D_Y(ST - TS). \label{derivation} \end{alignat}

Let $E_{ij}$  be the matrix with 1 in the $i, j$ position and 0 elsewhere and let $F_{ij} := D_Y(E_{ij})$. Then, for some fixed $c \in \N$,
$$f(Y) = \sum_{i = 1}^s \(E_{ii}F_{ii} + E_{i1}F_{1i}E_{ii}\) + cI_s$$
is a value for $f(Y)$ that satisfies equation \eqref{innerderivationproperty}, that is, $D_Y$ is an inner derivation. This is proven by plugging matrices of the form $E_{rs}$ and $E_{uv}$ into \eqref{derivation} for $S$ and $T$. This gives a large set of equalities which provide enough information to show that \eqref{innerderivationproperty} holds for all matrices $T$ of the form $E_{pq}$. It is then a simple matter to linearly extend this result to show that it holds for all matrices $T$. 

For higher order nc functions, we turn to undoing the ${_j\Delta} $ operators. That is, given $k + 1$ nc functions, $F_0, \ldots, F_k$, each of order $k + 1$, it is proved that there exists an nc function, $f$, of order $k$ such that ${_j\Delta} f = F_j$ for $0 \leq j \leq k$ if and only if ${_i\Delta} F_j = {_{j + 1}\Delta} F_i$ for $0 \leq i \leq j \leq k$.

First, an order $k$ nc function $f$ is defined up to a selection of the value of $f$ at some arbitrary  point $(Y^0, \ldots, Y^k) \in 
\Omega^{(0)}_{s_0} \times \hdots \times \Omega^{(k)}_{s_k}$. For $Z^j \in \N_j^{s_{j - 1}m_{j - 1} \times s_jm_j}$ where 
$j = 0, \ldots, k$, $f$ is defined at the ``amplified"  points  $(I_{m_0} \otimes Y^0, \ldots, I_{m_k} \otimes Y^k)$ as
\begin{multline*}
f(I_{m_0} \otimes Y^0, \ldots, I_{m_k} \otimes Y^k)(Z^1, \ldots, Z^k)\\
 = \[ \sum_{\underset{j = 1, \ldots, k - 1}{i_j = 1}}^{m_j} f(Y^0, \ldots, Y^k)(Z^1_{i_0, i_1}, \ldots, Z^k_{i_{k - 1}, i_k}) \]_{\begin{array}{ll}
i_0=1,\ldots,m_0,\\
i_k=1,\ldots,m_k
\end{array}
}.
\end{multline*}
Then, given $X^j \in \Omega^{(j)}_{s_jm_j}$ and $Z^j \in \N_j^{s_{j - 1}m_{j - 1} \times s_jm_j}$ for $j = 0, \ldots, k$,
\begin{alignat*}{2}
f(X^0, &\ldots, X^k)(Z^1, \ldots, Z^k) \\
&= f(I_{m_0} \otimes Y^0, \ldots, I_{m_k} \otimes Y^k)(Z^1, \ldots, Z^k) \\
&+ \sum_{j = 0}^k {_j\Delta} f(I_{m_0} \otimes Y^0, \ldots, I_{m_j} \otimes Y^j, X^j, \ldots, X^k)(Z^1, \ldots, Z^j, X^j - I_{m_j} \otimes Y^j, Z^{j + 1}, \ldots, Z^k). \end{alignat*}
It is then shown that this yields an nc function $f$ such that ${_j\Delta} f = F_j$, $j=0,\ldots,k$, if and only if there exists a value $f(Y^0, \ldots, Y^k)$ which, for appropriately sized matrices $T_0, \ldots, T_k$ over $\R$, satisfies
\begin{alignat}{2}
&\begin{aligned}
&T_0f(Y^0, \ldots, Y^k)(Z^1, \ldots, Z^k) - f(Y^0, \ldots, Y^k)(T_0Z^1, Z^2, \ldots, Z^k) \\
&\hspace{55mm}= F_0(Y^0, Y^0, Y^1, \ldots, Y^k)(T_0Y^0 - Y^0T_0, Z^1, \ldots, Z^k), \\ \end{aligned} \label{one} \\
&\begin{aligned}
&f(Y^0, \ldots, Y^k)(Z^1, \ldots, Z^{j - 1}, Z^jT_j, Z^{j + 1}, \ldots, Z^k) \\
&\hspace{55mm} - f(Y^0, \ldots, Y^k)(Z^1, \ldots, Z^j, T_jZ^{j + 1}, Z^{j + 2}, \ldots, Z^k) \\
&\hspace{15mm}= F_j(Y^0, \ldots, Y^{j - 1}, Y^j, Y^j, Y^{j + 1}, \ldots, Y^k)(Z^1, \ldots, Z^j, T_jY^j - Y^jT_j, Z^{j + 1}, \ldots, Z^k), \\ \end{aligned} \label{two} \\
&\begin{aligned}
&f(Y^0, \ldots, Y^k)(Z^1, \ldots, Z^{k - 1}, Z^kT_k) - f(Y^0, \ldots, Y^k)(Z^1, \ldots, Z^k)T_k \\
&\hspace{55mm}= F_k(Y^0, \ldots, Y^{k - 1}, Y^k, Y^k)(Z^1, \ldots, Z^k, T_kY^k - Y^kT_k). \end{aligned} \label{three} \end{alignat}

We let $${_jD}(S)=({_jD}_Y(S)) = F_j(Y^0, \ldots, Y^{j - 1}, Y^j, Y^j, Y^{j + 1}, \ldots, Y^k)(SY^j - Y^jS)$$ for $j = 0, \ldots, k$; this is viewed as a $k$-linear function on $Z^1, \ldots, Z^k$ by setting
\begin{alignat*}{2}
F_j&(Y^0, \ldots, Y^{j - 1}, Y^j, Y^j, Y^{j + 1}, \ldots, Y^k)(SY^j - Y^jS)(Z^1, \ldots, Z^k) \\
&= F_j(Y^0, \ldots, Y^{j - 1}, Y^j, Y^j, Y^{j + 1}, \ldots, Y^k)(Z^1, \ldots, Z^j, SY^j - Y^jS, Z^{j + 1}, \ldots, Z^k). \end{alignat*}
It is then shown that for each $F_j$, $j = 0, \ldots, k$, there exists a value $f_j(Y^0, \ldots, Y^k)$ that satisfies the corresponding difference formula, \eqref{one}, \eqref{two}, or \eqref{three}, by using
\begin{alignat}{3}\label{higher derivations}
\({_jD}(S)\)T - T\({_jD}(S)\) + S\({_jD}(T)\) - \({_jD}(T)\)S = {_jD}(ST - TS). \end{alignat}
Lastly, utilising the fact that ${_i\Delta} F_j = {_{j + 1}\Delta} F_i$, $i\le j$, it is shown that $f_0(Y^0, \ldots, Y^k)$, \ldots, $f_k(Y^0, \ldots, Y^k)$ can be chosen equal to one another so that we have a single value $f(Y^0, \ldots, Y^k)$ with which to define the function. 

The following is a brief outline of the paper. Section 1 goes through the details of the process of defining an antiderivative for a first order nc function. In Section 2, the details for antidifferentiating sets of higher order nc functions are provided. Finally, Section 3 specializes the integrability results in the following three important cases:
\begin{enumerate} 
\item The modules have the form $\R^d$.
\item The nc functions being integrated are nc polynomials.
\item $\R = \mathbb{C}$, and the nc functions are analytic.
\end{enumerate}

\section{Integrability of First Order NC Functions}

The main theorem of this section is the following.
\begin{theorem}\label{main}
Let $\Omega \subseteq \M_{\rm nc}$ be a right admissible nc set and let $F \in \T^1(\Omega, \Omega; \N_{\nc}, \M_{\nc})$. Then there exists an $f \in \T^0(\Omega; \N_{\nc})$ such that $\Delta f = F$ if and only if ${_0\Delta} F = {_1\Delta} F$.
Furthermore, $f$ is uniquely defined up to a scalar matrix $cI$ for $c \in \N$ in the following way: if $\tilde{f}$ is another nc function such that $\Delta \tilde{f} = F$, then
$$\tilde{f}(X) = f(X) + cI.$$
\end{theorem}

To prove this, we first demonstrate the following fact about nc sets.
\begin{lemma}\label{single dse}
Let $\Omega \subseteq \M_{\nc}$ be a right admissible nc set. Let $s$ be an integer such that $\Omega_s$ is nonempty. Then
$$\(\bigsqcup_{m = 1}^{\infty} \Omega_{sm}\)_{\rm d.s.e.} = \Omega_{\rm d.s.e.}.$$
\end{lemma}

\begin{proof}
It is clear that $\(\bigsqcup_{m = 1}^{\infty} \Omega_{sm}\)_{\rm d.s.e.} \subseteq \Omega_{\rm d.s.e.}$. For the reverse inclusion, let $X \in \Omega_{\rm d.s.e.}$; then $X \oplus Y \in \Omega_t$ for some $Y \in \M_{\nc}$. Let $\ell$ be the least common multiple of $s$ and $t$ and define $X \overline{\oplus} Y = \bigoplus_{i = 1}^{\ell/t} (X \oplus Y) \in \Omega_\ell$. Then $$X \overline{\oplus} Y =
X\oplus\left(Y\oplus  \bigoplus_{i = 1}^{\ell/t -1} (X \oplus Y)\right)
\in \bigsqcup_{m = 1}^{\infty} \Omega_{sm}.$$ 
So, by definition, $X \in \(\bigsqcup_{m = 1}^{\infty} \Omega_{sm}\)_{\rm d.s.e.}.$
\end{proof}

The following theorem is the first major step in the proof of our main result.
\begin{theorem}\label{almost result}
Let $\Omega \subseteq \M_{\nc}$ be a right admissible nc set, let $F \in \T^1(\Omega, \Omega; \N_{\nc}, \M_{\nc})$, and let $Y \in \Omega_s$. If ${_0\Delta} F = {_1\Delta} F$ and there exists an $f_0 \in \N^{s \times s}$ such that 
\begin{equation}\label{innerder}
F(Y, Y)(SY - YS) = Sf_0 - f_0S \end{equation}
for all $S \in \R^{s \times s}$, then there exists a unique $f \in \T^0(\Omega; \N_{\nc})$ such that $\Delta f = F$ and $f(Y) = f_0$. Furthermore, for $X \in \Omega_{sm}$, 
\begin{equation}\label{definition}
f(X) = I_m \otimes f_0 + F(I_m \otimes Y,X)(X - I_m \otimes Y). \end{equation}
\end{theorem}
\begin{proof}
Suppose that \eqref{innerder} holds for some $f_0 \in \N^{s \times s}$. Given $R \in \R^{sp \times sm}$, we note that \eqref{innerder} can be generalized using the direct sum property of $F$ to give the following:
$$F(I_p \otimes Y, I_m \otimes Y)(R(I_m \otimes Y) - (I_p \otimes Y)R) = R(I_m \otimes f_0) - (I_p \otimes f_0)R.$$
Secondly, it will be shown that for $X\in\Omega_{sm}$, $W \in \Omega_{sp}$, $R \in \R^{sp \times sm}$, and for $f(X)$, $f(W)$ defined as in \eqref{definition}, 
\begin{equation}\label{intertwine}
Rf(X) - f(W)R = F(W, X)(RX - WR).
\end{equation}
Indeed, using some of the difference formulas (3.40)--(3.46)  from \cite{Verb}, we obtain 
\begin{multline*}
Rf(X) - f(W)R = R(I_m \otimes f_0) +RF(I_m\otimes Y,X)(X-I_m\otimes Y)\\
- (I_p \otimes f_0)R - F(I_p \otimes Y,W)(W - I_p \otimes Y)R \\
= R(I_m \otimes f_0) - (I_p \otimes f_0)R\\
+RF(I_m\otimes Y,X)(X-I_m\otimes Y)- F(I_p \otimes Y,W)(W - I_p \otimes Y)R \\
=F(I_p\otimes Y,I_m\otimes Y)(R(I_m\otimes Y)-(I_p\otimes Y)R)\\
+F(W,X)((RX-R(I_m\otimes Y))+{_0}\Delta F(W,I_m\otimes Y,X)(R(I_m\otimes Y)-WR,X-I_m\otimes Y)\\
-F(I_p\otimes Y,I_m\otimes Y)(WR-(I_p\otimes Y)R)+{_1}\Delta F(I_p\otimes Y,W,I_m\otimes Y)(W-I_p\otimes Y,R(I_m\otimes Y)-WR)\\
=F(I_p\otimes Y,I_m\otimes Y)(R(I_m\otimes Y)-WR)+{_1}\Delta F(I_p\otimes Y,W,I_m\otimes Y)(W-I_p\otimes Y,R(I_m\otimes Y)-WR)\\
+{_0}\Delta F(W,I_m\otimes Y,X)(R(I_m\otimes Y)-WR,X-I_m\otimes Y)+F(W,X)((RX-R(I_m\otimes Y))\\
=F(I_p\otimes Y,I_m\otimes Y)(R(I_m\otimes Y)-WR)+{_0}\Delta F(I_p\otimes Y,W,I_m\otimes Y)(W-I_p\otimes Y,R(I_m\otimes Y)-WR)\\
+{_1}\Delta F(W,I_m\otimes Y,X)(R(I_m\otimes Y)-WR,X-I_m\otimes Y)+F(W,X)(RX-R(I_m\otimes Y))\\
=F(W,I_m\otimes Y)(R(I_m\otimes Y)-WR)+{_1}\Delta F(W,I_m\otimes Y,X)(R(I_m\otimes Y)-WR,X-I_m\otimes Y)\\
+F(W,X)(RX-R(I_m\otimes Y))\\
=F(W,X)(R(I_m\otimes Y)-WR)+F(W,X)((RX-R(I_m\otimes Y))\\
=F(W,X)(RX-WR).
 \end{multline*}

With this equality, it is clear now that $f$, as defined in \eqref{definition}, is an nc function on $\Omega' = \bigsqcup_{m = 1}^{\infty} \Omega_{sm}$ satisfying $f(Y)=f_0$. 

By Proposition 9.2 in \cite{Verb} and Lemma \ref{single dse}, it follows that $f$ can be uniquely extended to an nc function on $\Omega_{\rm d.s.e.}$. Thus, it is clear that the restriction of $f$ to $\Omega$ will also be an nc function.

Next, to see that $\Delta f = F$ on $\Omega'$, it will be shown that for $U \in \Omega_{sn}, V \in \Omega_{sq}, Z \in \R^{sn \times sq}$, $\Delta f(U, V)(Z) = F(U, V)(Z)$.

Let $W = \fmat{cc} U & Z \\ 0 & V \emat \in \Omega_{sp}$ where $p = n + q$. Let  $S \in \R^{sp \times sp}$ have the form 
$S = \fmat{cc} I_{sn} & 0 \\ 0 & 0 \emat.$
Thus
\begin{alignat*}{2}
Sf(W) - &f(I_p \otimes Y)S \\
&= \fmat{cc} I_{sn} & 0 \\ 0 & 0 \emat \fmat{cc} f(U) & \Delta f(U, V)(Z) \\ 0 & f(V) \emat - \fmat{cc} I_n \otimes f(Y) & 0 \\ 0 & I_q \otimes f(Y) \emat \fmat{cc} I_{sn} & 0 \\ 0 & 0 \emat \\
&= \fmat{cc} f(U) - I_n \otimes f(Y) & \Delta f(U, V)(Z) \\ 0 & 0 \emat. \end{alignat*}

At the same time,
\begin{alignat*}{2}
SW - (I_p \otimes Y)S &= \fmat{cc} I_{sn} & 0 \\ 0 & 0 \emat \fmat{cc} U & Z \\ 0 & V \emat - \fmat{cc} I_n \otimes Y & 0 \\ 0 & I_q \otimes Y \emat \fmat{cc} I_{sn} & 0 \\ 0 & 0 \emat \\
&= \fmat{cc} U - I_n \otimes Y & Z \\ 0 & 0 \emat. \end{alignat*}
As was proven above, $Sf(W) - f(I_p \otimes Y)S = F(I_p \otimes Y, W)(SW - (I_p \otimes Y)S)$. Hence,
\begin{alignat*}{2}
\fmat{cc} f(U) - I_n \otimes f(Y) & \Delta f(U, V)(Z) \\ 0 & 0 \emat &= F\left(I_p \otimes Y, \fmat{cc} U & Z \\ 0 & V \emat\right)\(\fmat{cc} U - I_n \otimes Y & Z \\ 0 & 0 \emat\). \end{alignat*}
Focusing only on the upper right hand entries gives us the following,
\begin{alignat*}{2}
\Delta f(U, V)(Z) &= {_1\Delta} F(I_n \otimes Y, U, V)(U - I_n \otimes Y, Z) + F(I_n \otimes Y, V)(Z) \\
\Delta f(U, V)(Z) &= F(U, V)(Z), \end{alignat*}
where the final step used the difference formula, equation (3.45) in \cite{Verb}, and the fact that ${_0\Delta} F = {_1\Delta} F$. Hence, $\Delta f(U, V)(Z) = F(U, V)(Z)$ as desired.

Now the functions $f$ and $F$ both have unique extensions to the set $\Omega_{\rm d.s.e.}$ by {\cite[Propositions 9.2, 9.3]{Verb}}. Suppose by contradiction that for some $U \in \Omega_r, V \in \Omega_t$ and $Z \in \R^{r \times t}$, $\Delta f(U, V)(Z) \ne F(U, V)(Z)$. For some integers $p$ and $q$, $U' = I_p \otimes U$ and $V' = I_q \otimes V$ are in $\Omega'$. Given $E$ as a $p \times q$ matrix of all ones, we would have that
$$\Delta f(U', V')(E \otimes Z) \ne F(U', V')(E \otimes Z).$$
But this is not the case. Hence $\Delta f = F$ on the entire set $\Omega_{\rm d.s.e.}$. Thus, the result clearly holds on the smaller set $\Omega$.

Finally, uniqueness of the nc function $f$ with the desired properties is a consequence of the fact that $f$ must satisfy \eqref{definition}.
\end{proof}

To prove Theorem \ref{main} it needs to be shown that a matrix $f_0$ will always exist. To do this some additional definitions and facts are needed.

\begin{definition}
\begin{enumerate} 
\item Let $E_{\alpha \beta}$ be the matrix with a $1$ in position $\alpha, \beta$ and zeros elsewhere.

\item Given a matrix $E_{\alpha\beta}$ and a function $F \in \T^1(\Omega, \Omega; \N_{\nc}, \M_{\nc})$, let $$F_{\alpha \beta}(Y) = F(Y, Y)(E_{\alpha \beta}Y - YE_{\alpha \beta}).$$
If the $Y$ is understood, this may be shortened to $F_{\alpha \beta}$.

\item A derivation of an algebra $\A$ is a linear map, $D\colon \A \to \N$, from $\A$ into a bimodule $\N$ over $\A$ that respects the Leibniz rule. That is, given $S, T \in \A$,
$$S D(T)+ D(S) T = D(ST).$$

\item A derivation of a Lie algebra $\A$ is a linear map, $D\colon \A \to \N$, from $\A$ into a bimodule $\N$ over $\A$ that respects the Leibniz rule. That is, given $S, T \in \A$,
$$[S, D(T)] + [D(S), T] = D([S, T]).$$

\item A derivation of a Lie algebra $\A$ is called inner if there exists an $N \in \N$ such that, for all $S \in \A$,
$$D(S) = [S,N].$$
\end{enumerate}
\end{definition}

Clearly, an algebra $\mathcal{A}$ gives rise to a Lie algebra $(\mathcal{A},[\cdot,\cdot])$ with the Lie bracket defined by $[S,T]=ST-TS$, and a derivation $D\colon \A\to\N$ is also a Lie-algebra derivation.

It turns out that we can restate some of our assumptions in terms of derivations to simplify the proof of the main theorem.
\begin{proposition}\label{derivationprop}
Let $\Omega \subseteq \M_{\nc}$ be a right admissible nc set, let $F \in \T^1(\Omega, \Omega; \N_{\nc}, \M_{\nc})$, let $Y \in \Omega_s$, and suppose that $${_0\Delta} F(Y,Y,Y) = {_1\Delta} F(Y,Y,Y).$$
Then the map
\begin{alignat*}{2}
D_Y &: \R^{s \times s} \to \N^{s \times s},
 \\
D_Y(S) &= F(Y, Y)(SY - YS) \end{alignat*}
is a derivation of the algebra $\R^{s \times s}$ with values in the bimodule $\N^{s \times s}$, and hence a Lie-algebra derivation.
\end{proposition}

\begin{proof}
It needs to be shown that
$$SD_Y(T)+ D_Y(S)T = D_Y(ST).$$
This is done as follows.
\begin{multline*}
 SD_Y(T) + D_Y(S)T
= SF(Y, Y)(TY - YT) 
+F(Y, Y)(SY - YS)T   \\
=\Big(SF(Y, Y)(TY - YT) - F(Y, Y)(S(TY - YT))\Big)\\
 - \Big(F(Y, Y)((SY - YS)T) - F(Y, Y)(SY-YS)T\Big)  
 +F(Y, Y)(STY - YST)\\
={_0{\Delta }}F(Y, Y, Y)(SY - YS, TY - YT)- {_1{\Delta }}F(Y, Y, Y)( SY - YS,TY-YT) 
+F(Y, Y)(STY-UST)\\
= F(Y, Y)(STY-YST) 
= D_Y(ST). \end{multline*}
\end{proof}

The following is a general result about derivations of Lie algebras.
\begin{theorem}\label{derivation properties}
Let $D\colon \R^{s \times s} \to \N^{s \times s}$ be a Lie algebra derivation and let $D^{i j} = D(E_{i j})$. Then $D^{i i}_{kk} = D^{i i}_{\ell\ell}$ for all $i, k, \ell = 1, \ldots, s$. Further, $D$ is inner, i.e. $D(S) = SN - NS$ for some $N\in\N^{s \times s}$, if and only if $D^{i i}_{k k} = 0$ for all $i, k = 1, \ldots, s$, and then 
\begin{equation}\label{definition inner derivation}
N = \sum_{i = 1}^n (E_{ii}D^{ii} + E_{i1}D^{1i}E_{ii}) + cI_s \end{equation}
for some $c \in \N$.
\end{theorem}

We will need the following lemma that gives the property of derivations of the Lie algebra $\R^{s\times s}$ in terms of the basis elements.
\begin{lemma} 
Let $D\colon \R^{s \times s} \to \N^{s \times s}$ be a Lie algebra derivation as in the formulation of Theorem \ref{derivation properties}. Then
\begin{alignat}{2}\label{Zero Structure}
E_{rs}D^{uv} - D^{uv}E_{rs} + D^{rs}E_{uv} - E_{uv}D^{rs} &= \left\{ \begin{array}{ccc} 0 & r \ne v & s \ne u \\ D^{rv} & r \ne v & s = u \\ - D^{us} & r = v & s \ne u \\ D^{rv} - D^{us} & r = v & s = u \end{array} \right. . \end{alignat}
\end{lemma}
\begin{proof}
Writing the defining property of the derivation $D$ with $S=E_{rs}$ and $T=E_{uv}$, we obtain
$$E_{rs}D^{uv} - D^{uv}E_{rs} + D^{rs}E_{uv} - E_{uv}D^{rs} =D(E_{rs}E_{uv}-E_{uv}E_{rs}),$$
and \eqref{Zero Structure} easily follows.
\end{proof}
\begin{proof}[Proof of Theorem \ref{derivation properties}]
Using the defining property of Lie algebra derivations and assuming that $i \ne j$, we have
\begin{alignat*}{2}
E_{j i}D(E_{i i}) - D(E_{i i})E_{j i} + D(E_{j i})E_{i i} - E_{ii}D(E_{j i}) &= D(E_{ji}E_{i i} - E_{ii}E_{j i}), \\
E_{ji}D^{i i} - D^{ii}E_{ji} + D^{j i}E_{i i} - E_{i i}D^{ji} &= D^{ji}. \end{alignat*}
Writing out the $j, i$ entry of these matrices gives
\begin{alignat*}{2}
D^{i i}_{i i} - D^{ii}_{j j} + D^{j i}_{ji} - 0 &= D^{j i}_{j i},\\
D^{i i}_{i i} - D^{i i}_{j j} &= 0, \\
D^{i i}_{i i} &= D^{ii}_{jj}. \end{alignat*}
Since $i, j$ are arbitrary and not equal, we find that  $D^{i i}_{kk} = D^{ii}_{\ell\ell}$ for all $i, k, \ell = 1, \ldots, s$ as desired.

For the second statement, let us first assume that $D$ is inner. In this case $D(S) = SN - NS$ for some $N \in \N^{s \times s}$, for all $S \in \R^{s \times s}$. Thus,
\begin{alignat*}{2}
D^{i i}_{kk} &= E_{k k}D(E_{ii})E_{k k} = E_{kk}(E_{i i}N - NE_{ii})E_{k k} \\
&= \left\{ \begin{array}{cc} 0-0,  & k \ne i, \\ N_{i i} - N_{i i}, & k = i,\end{array} \right. \\
&= 0. \end{alignat*}
To determine the form of $N$, we let $S = E_{ij}$. Then
\begin{alignat*}{2}
D^{ij} &= E_{ij}N - NE_{ij} \\
&= \fmat{ccc} & &  \\ N_{j1} & \hdots & N_{js} \\ & & \emat - \fmat{ccc} & N_{1i} & \\ & \vdots & \\ & N_{si} & \emat \\
&= \fmat{ccccc} & & N_{1i} & & \\ & & \vdots & & \\ N_{j1} & \hdots & N_{ii} - N_{jj} & \hdots & N_{js} \\ & & \vdots & & \\ & & N_{si} & & \emat, \end{alignat*}
where the last matrix has all entries that are not in the $i$-th row or $j$-th column equal to 0.
Hence $E_{ii}D^{ii}$ gives the values for row $i$ of $N$ except at position $i$ where it gives $0$. Further, it gives these values in row $i$. Thus, summing over $i$ gives all nondiagonal entries for $N$. Furthermore, $E_{i1}D^{1i}E_{ii}$ gives the value $N_{ii} - N_{11}$ for all $i$ and its puts it in position $i, i$. Thus, \eqref{definition inner derivation} holds with $c=N_{11}$.

Conversely, let $D^{ii}_{kk}=0$ for all $i,k=1,\ldots,s$, and let $N$ be given by \eqref{definition inner derivation}. Then
\begin{alignat*}{3}
E_{pq}& N - NE_{pq} \\
&= E_{pq}\( \sum_{i = 1}^n E_{ii}D^{ii} + E_{i1}D^{1i}E_{ii} + cI_s\) - \( \sum_{i = 1}^n E_{ii}D^{ii} + E_{i1}D^{1i}E_{ii} + cI_s\)E_{pq} \\
&= E_{pq}E_{qq}D^{qq} + E_{pq}E_{q1}D^{1q}E_{qq} + cE_{pq} - \(\sum_{i = 1}^n E_{ii}D^{ii}E_{pq} \) - E_{p1}D^{1p}E_{pp}E_{pq} - cE_{pq}\\
&= E_{pq}D^{qq} + E_{p1}D^{1q}E_{qq} - \(\sum_{i = 1}^n E_{ii}D^{ii}E_{pq} \) - E_{p1}D^{1p}E_{pq}. \end{alignat*}
Notice that if $p = q$ then we obtain
$$E_{pp}N - NE_{pp} = E_{pp}D^{pp} - \(\sum_{i = 1}^n E_{ii}D^{ii}E_{pp} \). $$
To complete the proof, it needs to be shown that the right-hand side (denoted by $X$ now) is equal to $D^{pq}$. Observe that various summands give us information about the rows or columns of $D^{ii}$, $D^{1p}$, and $D^{1q}$. Note that $E_{pq}D^{qq}$ tells us about row $p$ of $X$, $- \(\sum_{i = 1}^n E_{ii}D^{ii}E_{pq} \)$ tells us about column $q$ of $X$, and finally that $E_{p1}D^{1q}E_{qq} - E_{p1}D^{1p}E_{pq}$ tells us about entry $p, q$ of $X$. Thus, we will look at the equations for row $p$ except where row $p$ intersects column $q$; at column $q$ except where it intersects row $p$;  at the intersection point $p, q$. It must be shown that all other entries of the matrix $D^{pq}$ are $0$. Thus, for $X$ to be equal to $D^{pq}$, the following equalities must hold.
\begin{enumerate}
\item $D^{qq}_{qm} = D^{pq}_{pm}$ where $m \ne q$.
\item $- D^{\ell\ell}_{\ell p} = D^{pq}_{\ell q}$ where $\ell \ne p$.
\item $D^{qq}_{qq}+ D^{1q}_{1q} - D^{pp}_{pp}  - D^{1p}_{1p} = D^{pq}_{pq}$.
\item $0 = D^{pq}_{\ell m}$ where $\ell \ne p$ and $m \ne q$. 
\end{enumerate}
In the special case where $p=q$, the equalities become
\begin{enumerate}
\item $D^{pp}_{pm} = D^{pp}_{pm}$ where  $m \ne p$.
\item $- D^{\ell\ell}_{\ell p} = D^{pp}_{\ell p}$ where  $\ell \ne p$.
\item $0 = D^{pp}_{pp}$.
\item $0 = D^{pp}_{\ell m}$ where $\ell, m \ne p$. 
\end{enumerate}

We first prove these four equalities for the case where $p \ne q$.
\begin{enumerate}
\item In \eqref{Zero Structure}, let $r = p$ and $s = u = v = q$, and look at entry $p,m$. Then, when $m \ne q$, we obtain
$$D^{qq}_{qm} = D^{pq}_{pm}.$$
\item We prove the equality in two steps. The first when $\ell \ne q$ and the second when $\ell=q$. When $\ell \ne q$, apply \eqref{Zero Structure} with $r = s =\ell$, $u = p$,  and $v = q$. Entry $\ell, q$ tells us that
$$D^{pq}_{\ell q} + D^{\ell\ell}_{\ell p} = 0,$$
which implies the desired equality. If $\ell = q$, apply \eqref{Zero Structure} with $r = s = p$ and $u = v = q$. Entry $\ell, p$ tells us that
$$- D^{qq}_{qp} - D^{pp}_{qp} = 0 \implies - D^{qq}_{qp} = D^{pp}_{qp}.$$
Then apply \eqref{Zero Structure} with $r = u = v = p$ and $s = q$. Entry $q, q$ tells us that
$$- D^{pp}_{qp} = - D^{pq}_{qq} \implies D^{pq}_{qq} = D^{pp}_{qp}.$$
Thus,
$$- D^{qq}_{qp}
= D^{pq}_{qq}.$$
\item The equality will be handled in three steps. In the first, $p, q \ne 1$. In the second, $p = 1$ and in the third $q = 1$. Since $p \neq q$ this will cover all cases. Notice that in all these cases, $D^{qq}_{qq} - D^{pp}_{pp} = 0$ since we are assuming that diagonal entries are $0$.
\begin{enumerate}
\item When $p, q \ne 1$, apply \eqref{Zero Structure} with $r = v = p$, $s = q$ and $u = 1$. Then entry $1, q$ tells us that
$$- D^{1p}_{1p}-D^{pq}_{pq}  = -D^{1q}_{1q} \implies D^{1q}_{1q} - D^{1p}_{1p} = D^{pq}_{pq}.$$
\item If $q = 1$, apply \eqref{Zero Structure} with $r = v = p$, $s = u = 1$. Then entry $p, p$ tells us that
$$D^{1p}_{1p} + D^{p1}_{p1} = D^{pp}_{pp} - D^{11}_{pp} \implies  D^{11}_{11} - D^{1p}_{1p}=D^{p1}_{p1}.$$
\item Finally, if $p = 1$, we have
\begin{alignat*}{2}
D^{1q}_{1q} - D^{11}_{11} &= D^{1q}_{1q}. \end{alignat*}
This is immediate since $D^{11}_{11} = 0$.
\end{enumerate}
\item The equality will be handled in three steps.
\begin{enumerate}
\item First, apply  \eqref{Zero Structure} with $r = p$ and $s = u = v = q$. Then entry $\ell, m$ where $\ell \ne p, q$ and $m \ne q$ tells us that
$$0=D^{pq}_{\ell m}.$$
\item Second, apply  \eqref{Zero Structure} with $r = u = v = p$ and $s = q$. Then entry $\ell, m$ where $\ell = q$ and $m \ne p, q$ tells us that
$$-D^{pq}_{q m} = D^{pq}_{qm}.$$
Hence, $D^{pq}_{qm}=0$.
\item Finally, apply  \eqref{Zero Structure} with $r = u = p$, $s = q$ and $v \ne p, q$. Then entry $q, v$ tells us that
$$D^{pq}_{qp} = 0.$$
\end{enumerate}

\end{enumerate}
Now we consider the case $p = q$. 
\begin{enumerate}
\item The equality is immediate.
\item Apply \eqref{Zero Structure} with $r = s = p$ and $\ell=u = v$. Then entry $\ell, p$ tells us, for $\ell \ne p$,  that
$$-D^{\ell\ell}_{\ell p} - D^{pp}_{\ell p} = 0.$$ Hence $$-D^{\ell\ell}_{\ell p} = D^{pp}_{\ell p}.$$
\item This holds since we are assuming that diagonal elements are $0$.
\item Finally,  apply \eqref{Zero Structure} with $s = \ell$, $u = v = p$ and $r\ne p,\ell$. Then entry $r,m$ tells us, for $\ell \ne m,p$, $m\ne p$, that
$$D^{pp}_{\ell m} = 0.$$
If $\ell = m$, we have 
$$D^{pp}_{m m} = 0$$
by assumption.
\end{enumerate}

Thus, we have shown that $D$ is an inner derivation.
\end{proof}

We will need the following two propositions.
\begin{proposition}\label{deltas}
Let $\Omega^{(0)}\subseteq\M_{0, \nc}$, \ldots, $\Omega^{(k)}\subseteq\M_{k, \nc}$ be right admissible nc sets, and let $g\in\T^k(\Omega^{(0)},\ldots,\Omega^{(k)};\N_{0,\nc},\ldots,\N_{k,\nc})$. Then, for every $i,j$ such that $0\le i\le j\le k$,
\begin{equation}\label{deltadelta}
{_{j+1}}\Delta\,{_i}\Delta g={_i}\Delta\,{_j}\Delta g.
\end{equation}
\end{proposition}

Notice that, in the special case $i=j=k=0$, \eqref{deltadelta} becomes ${_1}\Delta\Delta g={_0}\Delta\Delta g$, which proves the ``only if" part of Theorem \ref{main}.
\begin{proof}[Proof of Proposition \ref{deltas}]
We will first present the proof for $i=j=k$. For any matrices $X^0\in\Omega_{n_0}^{(0)}$, \ldots, $X^{k}\in\Omega_{n_{k}}^{(k)}$, $X^{k+1}\in\Omega_{n_{k+1}}^{(k)}$, $X^{k+2}\in\Omega_{n_{k+2}}^{(k)}$, $Z^1\in\N_1^{n_0\times n_1}$, \ldots, $Z^k\in\N_k^{n_{k-1}\times n_k}$, $Z^{k+1}\in\M_k^{n_{k}\times n_{k+1}}$, $Z^{k+2}\in\M_k^{n_{k+1}\times n_{k+2}}$, we have
\begin{multline*}
{_{k+1}}\Delta\,{_k}\Delta g(X^0,\ldots,X^{k+2})(Z^1,\ldots,Z^{k+2})\\
={_k}\Delta g\left(X^0,\ldots,X^k,\begin{bmatrix}
X^{k+1} & Z^{k+2}\\
0 & X^{k+2}
\end{bmatrix}\right)(Z^1,\ldots,Z^k,\begin{bmatrix} Z^{k+1} & 0 \end{bmatrix})\begin{bmatrix}
0\\
I_{n_{k+2}}
\end{bmatrix}\\
= g\left(X^0,\ldots,X^{k-1},\begin{bmatrix}
X^k & Z^{k+1} & 0\\
0 & X^{k+1} & Z^{k+2}\\
0 & 0 & X^{k+2}
\end{bmatrix}\right)(Z^1,\ldots,Z^{k-1},\begin{bmatrix} Z^{k} & 0 & 0 \end{bmatrix})\begin{bmatrix}
0\\
0\\
I_{n_{k+2}}
\end{bmatrix}.
\end{multline*}
On the other hand,
\begin{multline*}
{_{k}}\Delta\,{_k}\Delta g(X^0,\ldots,X^{k+2})(Z^1,\ldots,Z^{k+2})\\
={_k}\Delta g\left(X^0,\ldots,X^{k-1},\begin{bmatrix}
X^{k} & Z^{k+1}\\
0 & X^{k+1}
\end{bmatrix},X^{k+2}\right)\left(Z^1,\ldots,Z^{k-1},\begin{bmatrix} Z^{k} & 0 \end{bmatrix},\begin{bmatrix}
0\\
Z^{k+2}
\end{bmatrix}\right)\\
= g\left(X^0,\ldots,X^{k-1},\begin{bmatrix}
X^k & Z^{k+1} & 0\\
0 & X^{k+1} & Z^{k+2}\\
0 & 0 & X^{k+2}
\end{bmatrix}\right)(Z^1,\ldots,Z^{k-1},\begin{bmatrix} Z^{k} & 0 & 0 \end{bmatrix})\begin{bmatrix}
0\\
0\\
I_{n_{k+2}}
\end{bmatrix},
\end{multline*}
which yields \eqref{deltadelta} in the case $i=j=k.$ 
In the cases $i=j<k$, the identity \eqref{deltadelta} is proved similarly.

Next we prove \eqref{deltadelta} for $i=0$, $j=k$. For any matrices $X^0\in\Omega_{n_0}^{(0)}$, $X^1\in\Omega_{n_1}^{(0)}$, $X^2\in\Omega_{n_2}^{(1)}$ \ldots, $X^{k}\in\Omega_{n_{k}}^{(k-1)}$, $X^{k+1}\in\Omega_{n_{k+1}}^{(k)}$, $X^{k+2}\in\Omega_{n_{k+2}}^{(k)}$, $Z^1\in\M_0^{n_0\times n_1}$, $Z^2\in\N_1^{n_1\times n_2}$ \ldots, $Z^k\in\N_{k-1}^{n_{k-1}\times n_k}$, $Z^{k+1}\in\N_k^{n_{k}\times n_{k+1}}$, $Z^{k+2}\in\M_k^{n_{k+1}\times n_{k+2}}$, we have
\begin{multline*}
{_{k+1}}\Delta\,{_0}\Delta g(X^0,\ldots,X^{k+2})(Z^1,\ldots,Z^{k+2})\\
={_0}\Delta g\left(X^0,\ldots,X^k,\begin{bmatrix}
X^{k+1} & Z^{k+2}\\
0 & X^{k+2}
\end{bmatrix}\right)(Z^1,\ldots,Z^k,\begin{bmatrix} Z^{k+1} & 0 \end{bmatrix})\begin{bmatrix}
0\\
I_{n_{k+2}}
\end{bmatrix}\\
=\begin{bmatrix}
I_{n_0} & 0
\end{bmatrix} g\left(\begin{bmatrix}
X^{0} & Z^{1}\\
0 & X^{1}
\end{bmatrix},X^2,\ldots,X^{k},\begin{bmatrix}
X^{k+1} & Z^{k+2}\\
0 & X^{k+2}
\end{bmatrix}\right)\left(\begin{bmatrix}
0\\ Z^2
\end{bmatrix},
Z^3,\ldots,Z^{k},\begin{bmatrix} Z^{k+1} & 0 \end{bmatrix}\right)\begin{bmatrix}
0\\
I_{n_{k+2}}
\end{bmatrix}.
\end{multline*}
On the other hand,
\begin{multline*}
{_{0}}\Delta\,{_k}\Delta g(X^0,\ldots,X^{k+2})(Z^1,\ldots,Z^{k+2})\\
=\begin{bmatrix}
I_{n_0} & 0
\end{bmatrix}{_k}\Delta g\left(\begin{bmatrix}
X^{0} & Z^{1}\\
0 & X^{1}
\end{bmatrix},X^2,\ldots,X^{k+2}\right)\left(\begin{bmatrix}
0\\ Z^2
\end{bmatrix},Z^3,\ldots,Z^{k+2}\right)\\
= \begin{bmatrix}
I_{n_0} & 0
\end{bmatrix} g\left(\begin{bmatrix}
X^{0} & Z^{1}\\
0 & X^{1}
\end{bmatrix},X^2,\ldots,X^{k},\begin{bmatrix}
X^{k+1} & Z^{k+2}\\
0 & X^{k+2}
\end{bmatrix}\right)\left(\begin{bmatrix}
0\\ Z^2
\end{bmatrix},
Z^3,\ldots,Z^{k},\begin{bmatrix} Z^{k+1} & 0 \end{bmatrix}\right)\begin{bmatrix}
0\\
I_{n_{k+2}}
\end{bmatrix},
\end{multline*}
which yields \eqref{deltadelta} in the case $i=0$, $j=k.$ 

In the other cases where $i<j$, the identity \eqref{deltadelta} is proved similarly.
\end{proof}

\begin{proposition}\label{makingzero}
Let $\Omega \subseteq \M_{\nc}$ be a right admissible nc set,  $F \in \T^1(\Omega, \Omega; \N_{\nc}, \M_{\nc})$, and $Y \in \Omega_s$. Suppose $${_0}\Delta F(Y,Y,Y)= {_1}\Delta F(Y,Y,Y).$$  Let $A, B, C \in \R^{s \times s}$ such that 
$$AB = \lambda A \text{ and } BC = \lambda C \text{ for some } \lambda \in \R.$$
Then $$AF(Y, Y)(BY - YB)C = 0.$$
In particular, if $A = B = C = P$ and $P^2 = P$, then  $$PF(Y, Y)(PY - YP)P = 0.$$
\end{proposition}

\begin{proof}
\allowdisplaybreaks
\begin{alignat*}{2}
A&F(Y, Y)(BY - YB)C \\
&= F(Y, Y)(ABY - AYB)C + {_0\Delta} F(Y, Y, Y)(AY - YA, BY - YB)C \\
&= F(Y, Y)(ABYC - AYBC) - {_1\Delta} F(Y, Y, Y)(ABY - AYB, CY - YC) \\
&\hspace{10mm}+ {_0\Delta} F(Y, Y, Y)(AY - YA, BYC - YBC) \\
&\hspace{10mm}- {_2\Delta}\,{_0\Delta} F(Y, Y, Y, Y)(AY - YA, BY - YB, CY - YC) \\
&= F(Y, Y)(\lambda AYC - \lambda AYC) - {_1\Delta} F(Y, Y, Y)(\lambda AY - AYB, CY - YC) \\
&\hspace{10mm}+ {_0\Delta} F(Y, Y, Y)(AY - YA, BYC - \lambda YC) \\
&\hspace{10mm}- {_2\Delta}\,{_0\Delta} F(Y, Y, Y, Y)(AY - YA, BY - YB, CY - YC) \\
&= - {_1\Delta} F(Y, Y, Y)(\lambda YA - AYB, CY - YC) - \lambda\,{_1\Delta} F(Y, Y, Y)(AY - YA, CY - YC) \\
&\hspace{10mm}+ \lambda\,{_0\Delta} F(Y, Y, Y)(AY - YA, CY - YC) + {_0\Delta} F(Y, Y, Y)(AY - YA, BYC - \lambda CY) \\
&\hspace{10mm}- {_2\Delta}{_0\Delta} F(Y, Y, Y, Y)(AY - YA, BY - YB, CY - YC) \\
&= -{_1\Delta} F(Y, Y, Y)(YAB - AYB, CY - YC)  + {_0\Delta} F(Y, Y, Y)(AY - YA, BYC - BCY) \\
&\hspace{10mm}- {_2\Delta}{_0\Delta} F(Y, Y, Y, Y)(AY - YA, BY - YB, CY - YC) \\
&= {_0\Delta} F(Y, Y, Y)((AY - YA)B, CY - YC) - {_0\Delta} F(Y, Y, Y)(AY - YA, B(CY - YC)) \\
&\hspace{10mm}- {_2\Delta}{_0\Delta} F(Y, Y, Y, Y)(AY - YA, BY - YB, CY - YC) \\
&= {_1\Delta} {_0\Delta} F(Y, Y, Y, Y)(AY - YA, BY - YB, CY - YC) \\
&\hspace{10mm}- {_2\Delta}{_0\Delta} F(Y, Y, Y, Y)(AY - YA, BY - YB, CY - YC) \\
&= {_1\Delta} {_1\Delta} F(Y, Y, Y, Y)(AY - YA, BY - YB, CY - YC) \\
&\hspace{10mm}- {_2\Delta}{_1\Delta} F(Y, Y, Y, Y)(AY - YA, BY - YB, CY - YC) \\
&=0, \end{alignat*}
where the last equality follows from Proposition \ref{deltas}.
\end{proof}

\begin{corollary}\label{inner derivation}
Let $\Omega \subseteq \M_{\nc}$ be a right admissible nc set, let $F \in \T^1(\Omega, \Omega; \N_{\nc}, \M_{\nc})$ and let $Y \in \Omega_s$. Then $D_Y$ is an inner derivation. 
\end{corollary}

\begin{proof}
It only needs to be shown that ${(D_Y)}^{ii}_{kk} = 0$ for all $i, k = 1, \ldots, s$, and then Proposition \ref{derivationprop} and Theorem \ref{derivation properties} yield the result. It follows  from Proposition \ref{makingzero} with $P = E_{ii}$ that ${(D_Y)}^{ii}_{ii} = 0$. The first statement of Theorem \ref{derivation properties} then shows that ${(D_Y)}^{ii}_{kk} = 0$ for all $i, k = 1, \ldots, s$.
\end{proof}

We note that the derivation $D$ in Theorem \ref{derivation properties} is automatically inner and thus Corollary \ref{inner derivation} easily follows in the case where $\R=k$ is a field of characteristic 0 and $\N$ is a finite-dimensional vector space over $k$ by the Zassenhaus theorem (see, e.g., \cite[Theorem 6]{Jac}).

We are now in a position to finish the proof of the main theorem.
\begin{proof}[Proof of Theorem \ref{main}]  
First note that if there exists an $f \in \T^0(\Omega; \N)$ such that $\Delta f = F$, then Theorem 3.24 in \cite{Verb} shows that ${_0\Delta} F = {_1\Delta} F$. As we have mentioned earlier, this is also a special case, $k=0$, of Proposition \ref{deltas}.  For the converse, Corollary \ref{inner derivation} shows that $D_Y$ is inner, so that for any fixed $Y$ there exists an $f_0$ such that \eqref{innerder} holds for all $S \in \R^{s \times s}$. Moreover, it follows from Theorem \ref{derivation properties} with $D = D_Y$, that $f_0 = N$ where $N$ is defined by \eqref{definition inner derivation}. This satisfies the requirements of Theorem \ref{almost result} which then guarantees the existence of a function $f \in \T^0(\Omega; \N_{\nc})$ such that $\Delta f = F$. Thus, the proof is complete.
\end{proof}

\section{Integrability of Higher Order NC Functions}

The main result of this section is the following theorem extending Theorem \ref{main} to higher order nc functions.
\begin{theorem}\label{higher-main}
Let $\Omega^{(0)}, \ldots, \Omega^{(k)}$ be right admissible nc sets.  For $j = 0, \ldots, k$, let $$F_j \in \T^{k + 1}(\Omega^{(0)}, \ldots, \Omega^{(j - 1)}, \Omega^{(j)}, \Omega^{(j)}, \Omega^{(j + 1)}, \ldots, \Omega^{(k)}; \N_{0, \nc}, \ldots, \N_{j, \nc}, \M_{j, \nc}, \N_{j + 1, \nc}, \ldots, \N_{k, \nc}).$$ Then there exists an $f \in \T^k(\Omega^{(0)}, \ldots, \Omega^{(k)}; \N_{0, \nc}, \ldots, \N_{k, \nc})$ such that $${_j\Delta} f = F_j,\quad j = 0, \ldots, k,$$ if and only if $${_i\Delta} F_j = {_{j + 1}\Delta} F_i,\quad 0 \leq i \leq j \leq k.$$
Further, $f$ is uniquely determined up to a $k$-linear mapping  $c\colon \N_1\times \cdots \times\N_k \to \N_0$. Thus, if $\tilde{f}$ is another antiderivative, then
$$\tilde{f}(X^0, \ldots, X^k)(Z^1, \ldots, Z^k) = f(X^0, \ldots, X^k)(Z^1, \ldots, Z^k) + C(Z^1, \ldots, Z^k),$$
where, for $X^j\in\Omega^{(j)}_{n_j}$ and $Z^j\in\N_j^{n_{j-1}\times n_j}$, one has 
\begin{equation}\label{C}
C(Z^1, \ldots, Z^k) = \[ \sum_{\underset{j = 1, \ldots, k - 1}{\alpha_j = 1}}^{n_j} c(Z^1_{\alpha_0, \alpha_1}, \ldots, Z^k_{\alpha_{k - 1}, \alpha_k}) \]_{\begin{array}{ll}
\alpha_0=1,\ldots,n_0,\\
\alpha_k=1,\ldots,n_k
\end{array}}.
\end{equation}
\end{theorem}

We first extend Lemma \ref{single dse} to Cartesian products of nc sets.
\begin{lemma}\label{higher dse}
Let $\Omega^{(j)} \subseteq \M_{j, \nc}$ be right admissible nc sets for $j = 0, \ldots, k$. Let $s_0, \ldots, s_k$ be integers such that $\Omega^{(0)}_{s_0}, \ldots, \Omega^{(j)}_{s_k}$ are nonempty. Then
$$\(\bigsqcup_{m_0 = 1}^{\infty} \Omega^{(0)}_{s_0m_0} \)_{\rm d.s.e.} \times \hdots \times \(\bigsqcup_{m_k = 1}^{\infty} \Omega^{(k)}_{s_km_k} \)_{\rm d.s.e.} = \Omega^{(0)}_{\rm d.s.e} \times \hdots \times \Omega^{(k)}_{\rm d.s.e.}$$
\end{lemma}

\begin{proof}
By Lemma \ref{single dse}, corresponding elements of the direct products are equal. Hence, both quantities are necessarily equal.
\end{proof}

The following theorem generalizes Theorem \ref{almost result} to higher order nc functions.
\begin{theorem}\label{almost result again}
Let $\Omega^{(j)} \subseteq \M_{j, \nc}$ be right admissible nc sets, let $Y^j \in \Omega^{(j)}_{s_j}$,
and let $$F_j \in \T^{k + 1}(\Omega^{(0)}, \ldots, \Omega^{(j - 1)}, \Omega^{(j)}, \Omega^{(j)}, \Omega^{(j + 1)}, \ldots, \Omega^{(k)}; \N_{0, \nc}, \ldots, \N_{j, \nc}, \M_{j, \nc}, \N_{j + 1, \nc}, \ldots, \N_{k, \nc})$$
satisfy ${_i\Delta} F_j = {_{j + 1}\Delta} F_i$ for $0\le i\le j\le k$. Suppose there exists $$g \in \hom_{\R}(\N_1^{s_0 \times s_1} \otimes \hdots \otimes \N_k^{s_{k - 1} \times s_k}, \N_0^{s_0 \times s_k})$$ such that 
\begin{alignat}{2}\label{defining formula j equals 0}
&\begin{aligned}
F_0&(Y^0, Y^0, Y^1, \ldots, Y^k)(R_0Y^0 - Y^0R_0, Z^1, \ldots, Z^k) \\
&= R_0g(Z^1, \ldots, Z^k) - g(R_0Z^1, Z^2, \ldots, Z^k),\quad  R_0 \in \R^{s_0 \times s_0},\end{aligned} \end{alignat}
\begin{alignat}{2}\label{defining formula j greater 0}
&\begin{aligned}
F_j&(Y^0, \ldots, Y^{j - 1}, Y^j, Y^j, Y^{j + 1}, \ldots, Y^k)(Z^1, \ldots, Z^j, R_jY^j - Y^jR_j, Z^{j + 1}, \ldots, Z^k) \\
&= g(Z^1, \ldots, Z^{j - 1}, Z^jR_j, Z^{j + 1}, \ldots, Z^k) - g(Z^1, \ldots, Z^{j - 1}, Z^j, R_jZ^{j + 1}, \ldots, Z^k),\quad R_j \in 
\R^{s_j \times s_j},  \end{aligned} \end{alignat}
for all $j=1,\ldots,k-1$, and 
\begin{alignat}{2}\label{defining formula j equals k}
&\begin{aligned}
F_k&(Y^0, \ldots, Y^{j - 1}, Y^k, Y^k)(Z^1, \ldots, Z^k, R_kY^k - Y^kR_k) \\
&= g(Z^1, \ldots, Z^{k - 1}, Z^kR_k) - g(Z^1, \ldots, Z^{k - 1}, Z^k)R_k, \quad R_k \in \R^{s_k \times s_k}.\end{aligned} \end{alignat}
Then there exists $f \in \T^k(\Omega^{(0)}, \ldots, \Omega^{(k)}; \N_{0, \nc}, \ldots, \N_{k, \nc})$ such that ${_j\Delta} f = F_j$, $j=0,\ldots,k$. Furthermore, for $X^j \in \Omega^{(j)}_{s_jm_j}$, $j = 0, \ldots, k$, one has
\begin{multline}\label{higher-f}
f(X^0, \ldots, X^k)(Z^1, \ldots, Z^k)
= G(Z^1, \ldots, Z^k) \\
+ \sum_{j = 0}^k F_j(I_{m_0} \otimes Y^0, \ldots, I_{m_j} \otimes Y^j,X^j, \ldots, X^k)(Z^1, \ldots, Z^j, X^j - I_{m_j} \otimes Y^j, Z^{j + 1}, \ldots, Z^k), \end{multline}
where
\begin{equation*}
G(Z^1, \ldots, Z^k)
 = \[ \sum_{\underset{j = 1, \ldots, k - 1}{i_j = 1}}^{m_j} g(Z^1_{i_0, i_1}, \ldots, Z^k_{i_{k - 1}, i_k}) \]_{\begin{array}{ll}
i_0=1,\ldots,m_0,\\
i_k=1,\ldots,m_k
\end{array}
}.
\end{equation*}
\end{theorem}

\begin{proof}
First, note that formulas \eqref{defining formula j equals 0}--\eqref{defining formula j equals k} can be modifed by the direct sum rule using  $R_j \in \R^{s_jp_j \times s_jm_j}$ for $j = 1, \ldots, k$ to give the following:
\begin{alignat*}{2}
F_0&(I_{p_0} \otimes Y^0, I_{m_0} \otimes Y^0, I_{m_1} \otimes Y^1, \ldots, I_{m_k} \otimes Y^k)(R_0(I_{m_0} \otimes Y^0) - (I_{p_0} \otimes Y^0)R_0, Z^1, \ldots, Z^k) \\
&= R_0G(Z^1, \ldots, Z^k) - G(R_0Z^1, Z^2, \ldots, Z^k), \\
F_j&(I_{m_0} \otimes Y^0, \ldots, I_{m_{j - 1}} \otimes Y^{j - 1}, I_{p_j} \otimes Y^j, I_{m_j} \otimes Y^j, \ldots, I_{m_k} \otimes Y^k) \\ 
&\hspace{50mm} (Z^1, \ldots, Z^j, R_j(I_{m_j} \otimes Y^j) - (I_{p_j} \otimes Y^j)R_j, Z^{j + 1}, \ldots, Z^k) \\
&= G(Z^1, \ldots, Z^{j - 1}, Z^jR_j, Z^{j + 1}, \ldots, Z^k) - G(Z^1, \ldots, Z^j, R_jZ^{j + 1}, Z^{j + 2}, \ldots, Z^k),\ j=1,\ldots,k-1, \\
F_k&(I_{m_0} \otimes Y^0, \ldots, I_{m_{k - 1}} \otimes Y^{k - 1}, I_{p_k} \otimes Y^k, I_{m_k} \otimes Y^k)(Z^1, \ldots, Z^k, R_k(I_{m_k} \otimes Y^k) - (I_{p_k} \otimes Y^k)R_k) \\
&= G(Z^1, \ldots, Z^{k - 1}, Z^kR_k)  - G(Z^1, \ldots, Z^k)R_k. \end{alignat*}

Next, it will be shown that, for $X^j\in\Omega^{(j)}_{sm_j}$, $W^j\in\Omega^{(j)}_{sp_j}$, $R_j \in \R^{s_jp_j \times s_jm_j}$, and for $f(X)$, $f(W)$ defined as in \eqref{higher-f},
\begingroup
\allowdisplaybreaks
\begin{alignat*}{2}
&R_0f(X^0,\ldots,X^k)(Z^1R_1, \ldots, Z^kR_k) -
f(W^0, \ldots, W^k)(R_0Z^1, \ldots, R_{k-1}Z^k)R_k \\
&\hspace{15mm} = \sum_{i = 0}^k F_i(W^0, \ldots, W^i,X^i, \ldots, X^k)(Z^1, \ldots, Z^i, R_iX^i - W^iR_i, Z^{i + 1}, \ldots, Z^k).
 \end{alignat*} \endgroup
Indeed, using the difference formulas (3.40)--(3.46) from \cite{Verb}, we obtain
\allowdisplaybreaks
\begin{alignat*}{2}
&R_0f(X^0,\ldots,X^k)(Z^1R_1, \ldots, Z^kR_k) -
f(W^0, \ldots, W^k)(R_0Z^1, \ldots, R_{k-1}Z^k)R_k \\
&\hspace{10mm}=R_0G(Z^1R_1, \ldots, Z^kR_k)\\
&\hspace{20mm}+R_0\sum_{j = 0}^k F_j(I_{m_0} \otimes Y^0,\ldots, I_{m_j} \otimes Y^j,X^j, \ldots, X^k)\\
&\hspace{40mm}(Z^1R_1, \ldots, Z^jR_j, X^j -I_{m_j} \otimes Y^j, Z^{j + 1}R_{j+1}, \ldots, Z^kR^k)\\
&-G(R_0Z^1, \ldots, R_{k-1}Z^k)R_k \\
&\hspace{20mm}-\sum_{j = 0}^k F_j(I_{p_0} \otimes Y^0,\ldots, I_{p_j} \otimes Y^j,W^j, \ldots, W^k)\\
&\hspace{40mm}(R_0Z^1, \ldots, R_{j-1}Z^j, W^j -I_{p_j} \otimes Y^j, R_jZ^{j + 1}, \ldots, R_{k-1}Z^k)R_k\\
&=R_0G(Z^1R_1, \ldots, Z^kR_k)-G(R_0Z^1, \ldots, R_{k-1}Z^k)R_k \\
&+\sum_{j = 0}^k R_0F_j(I_{m_0} \otimes Y^0,\ldots, I_{m_j} \otimes Y^j,X^j, \ldots, X^k)\\
&\hspace{40mm}
(Z^1R_1, \ldots, Z^jR_j, X^j -I_{m_j} \otimes Y^j, Z^{j + 1}R_{j+1}, \ldots, Z^kR^k)\\
&-\sum_{j = 0}^k F_j(I_{p_0} \otimes Y^0,\ldots, I_{p_j} \otimes Y^j,W^j, \ldots, W^k)\\
&\hspace{40mm}(R_0Z^1, \ldots, R_{j-1}Z^j, W^j -I_{p_j} \otimes Y^j, R_jZ^{j + 1}, \ldots, R_{k-1}Z^k)R_k\\
&=\sum_{j = 0}^k F_j(I_{p_0} \otimes Y^0,\ldots, I_{p_j} \otimes Y^j,I_{m_j} \otimes Y^j,\ldots, I_{m_k} \otimes Y^k)\\
&\hspace{20mm}
(R_0Z^1, \ldots, R_{j-1}Z^j,R_j(I_{m_j} \otimes Y^j)-(I_{p_j} \otimes Y^j)R_j,  Z^{j + 1}R_{j+1}, \ldots, Z^kR_k)\\
&+\sum_{j = 0}^k F_j(W^0,\ldots, W^j,X^j, \ldots, X^k)\\
&\hspace{40mm}
(R_0Z^1, \ldots, R_{j-1}Z^j, R_jX^j -R_j(I_{m_j} \otimes Y^j), Z^{j + 1}R_{j+1}, \ldots, Z^kR_k)\\
&+\sum_{j = 0}^k\sum_{i=0}^j {_i}\Delta F_j(W^0,\ldots, W^i,I_{m_i} \otimes Y^i,\ldots, I_{m_j} \otimes Y^j,X^j, \ldots, X^k)\\
&
(R_0Z^1, \ldots, R_{i-1}Z^i, R_i(I_{m_i} \otimes Y^i)-W^iR_i, Z^{i + 1}R_{i+1}, \ldots, Z^jR_j,X^j-I_{m_j} \otimes Y^j,Z^{j + 1}R_{j+1}, \ldots, Z^kR_k)\\
&-\sum_{j = 0}^k F_j(I_{p_0} \otimes Y^0,\ldots, I_{p_j} \otimes Y^j,I_{m_j} \otimes Y^j,\ldots, I_{m_k} \otimes Y^k)\\
&\hspace{20mm}(R_0Z^1, \ldots, R_{j-1}Z^j, W^jR_j -(I_{p_j} \otimes Y^j)R_j,Z^{j + 1} R_{j+1}, \ldots, Z^kR_k)\\
&+\sum_{j = 0}^k\sum_{i=j}^k {_{i+1}}\Delta F_j(I_{p_0} \otimes Y^0,\ldots, I_{p_j} \otimes Y^j,W^j,\ldots, W^i,I_{m_i} \otimes Y^i,\ldots, I_{m_k} \otimes Y^k)\\
&
(R_0Z^1, \ldots, R_{j-1}Z^j, W^j-I_{p_j} \otimes Y^j,  R_jZ^{j+1}, \ldots, R_{i-1}Z^{i}, R_i(I_{m_i} \otimes Y^i)-W^iR_i,Z^{i + 1}R_{i+1}, \ldots, Z^kR_k)\\
&=\sum_{j = 0}^k F_j(I_{p_0} \otimes Y^0,\ldots, I_{p_j} \otimes Y^j,I_{m_j} \otimes Y^j,\ldots, I_{m_k} \otimes Y^k)\\
&\hspace{20mm}
(R_0Z^1, \ldots, R_{j-1}Z^j,R_j(I_{m_j} \otimes Y^j)-W^jR_j,  Z^{j + 1}R_{j+1}, \ldots, Z^kR_k)\\
&+\sum_{j = 0}^k\sum_{i=j}^k {_{i+1}}\Delta F_j(I_{p_0} \otimes Y^0,\ldots, I_{p_j} \otimes Y^j,W^j,\ldots, W^i,I_{m_i} \otimes Y^i,\ldots, I_{m_k} \otimes Y^k)\\
&
(R_0Z^1, \ldots, R_{j-1}Z^j, W^j-I_{p_j} \otimes Y^j,  R_jZ^{j+1}, \ldots, R_{i-1}Z^{i}, R_i(I_{m_i} \otimes Y^i)-W^iR_i,Z^{i + 1}R_{i+1}, \ldots, Z^kR_k)\\
&+\sum_{j = 0}^k\sum_{i=0}^j {_i}\Delta F_j(W^0,\ldots, W^i,I_{m_i} \otimes Y^i,\ldots, I_{m_j} \otimes Y^j,X^j, \ldots, X^k)\\
&
(R_0Z^1, \ldots, R_{i-1}Z^i, R_i(I_{m_i} \otimes Y^i)-W^iR_i, Z^{i + 1}R_{i+1}, \ldots, Z^jR_j,X^j-I_{m_j} \otimes Y^j,Z^{j + 1}R_{j+1}, \ldots, Z^kR_k)\\
&+\sum_{j = 0}^k F_j(W^0,\ldots, W^j,X^j, \ldots, X^k)\\
&\hspace{40mm}
(R_0Z^1, \ldots, R_{j-1}Z^j, R_jX^j -R_j(I_{m_j} \otimes Y^j), Z^{j + 1}R_{j+1}, \ldots, Z^kR_k)\\
&=\sum_{i = 0}^k F_i(I_{p_0} \otimes Y^0,\ldots, I_{p_i} \otimes Y^i,I_{m_i} \otimes Y^i,\ldots, I_{m_k} \otimes Y^k)\\
&\hspace{20mm}
(R_0Z^1, \ldots, R_{i-1}Z^i,R_i(I_{m_i} \otimes Y^i)-W^iR_i,  Z^{i + 1}R_{i+1}, \ldots, Z^kR_k)\\
&+\sum_{i = 0}^k\sum_{j=0}^i {_{j}}\Delta F_i(I_{p_0} \otimes Y^0,\ldots, I_{p_j} \otimes Y^j,W^j,\ldots, W^i,I_{m_i} \otimes Y^i,\ldots, I_{m_k} \otimes Y^k)\\
&
(R_0Z^1, \ldots, R_{j-1}Z^j, W^j-I_{p_j} \otimes Y^j,  R_jZ^{j+1}, \ldots, R_{i-1}Z^{i}, R_i(I_{m_i} \otimes Y^i)-W^iR_i,Z^{i + 1}R_{i+1}, \ldots, Z^kR_k)\\
&+\sum_{i = 0}^k\sum_{j=i}^k {_{j+1}}\Delta F_i(W^0,\ldots, W^i,I_{m_i} \otimes Y^i,\ldots, I_{m_j} \otimes Y^j,X^j, \ldots, X^k)\\
&
(R_0Z^1, \ldots, R_{i-1}Z^i, R_i(I_{m_i} \otimes Y^i)-W^iR_i, Z^{i + 1}R_{i+1}, \ldots, Z^jR_j,X^j-I_{m_j} \otimes Y^j,Z^{j + 1}R_{j+1}, \ldots, Z^kR_k)\\
&+\sum_{i = 0}^k F_i(W^0,\ldots, W^i,X^i, \ldots, X^k)
(R_0Z^1, \ldots, R_{i-1}Z^i, R_iX^i -R_i(I_{m_i} \otimes Y^i), Z^{i + 1}R_{i+1}, \ldots, Z^kR_k)\\
&=\sum_{i = 0}^k F_i(W^0,\ldots, W^i,I_{m_i} \otimes Y^i,\ldots, I_{m_k} \otimes Y^k)\\
&\hspace{20mm}
(R_0Z^1, \ldots, R_{i-1}Z^i,R_i(I_{m_i} \otimes Y^i)-W^iR_i,  Z^{i + 1}R_{i+1}, \ldots, Z^kR_k)\\
&+\sum_{i = 0}^k\sum_{j=i}^k {_{j+1}}\Delta F_i(W^0,\ldots, W^i,I_{m_i} \otimes Y^i,\ldots, I_{m_j} \otimes Y^j,X^j, \ldots, X^k)\\
&
(R_0Z^1, \ldots, R_{i-1}Z^i, R_i(I_{m_i} \otimes Y^i)-W^iR_i, Z^{i + 1}R_{i+1}, \ldots, Z^jR_j,X^j-I_{m_j} \otimes Y^j,Z^{j + 1}R_{j+1}, \ldots, Z^kR_k)\\
&+\sum_{i = 0}^k F_i(W^0,\ldots, W^i,X^i, \ldots, X^k)
(R_0Z^1, \ldots, R_{i-1}Z^i, R_iX^i -R_i(I_{m_i} \otimes Y^i), Z^{i + 1}R_{i+1}, \ldots, Z^kR_k)\\
&=\sum_{i = 0}^k F_i(W^0,\ldots, W^i,X^i,\ldots, X^k)
(R_0Z^1, \ldots, R_{i-1}Z^i,R_i(I_{m_i} \otimes Y^i)-W^iR_i,  Z^{i + 1}R_{i+1}, \ldots, Z^kR_k)\\
&+\sum_{i = 0}^k F_i(W^0,\ldots, W^i,X^i, \ldots, X^k)
(R_0Z^1, \ldots, R_{i-1}Z^i, R_iX^i -R_i(I_{m_i} \otimes Y^i), Z^{i + 1}R_{i+1}, \ldots, Z^kR_k)\\
&=\sum_{i = 0}^k F_i(W^0,\ldots, W^i,X^i,\ldots, X^k)
(R_0Z^1, \ldots, R_{i-1}Z^i,R_iX^i-W^iR_i,  Z^{i + 1}R_{i+1}, \ldots, Z^kR_k).
\end{alignat*}

With this equality, it is clear now that $f$, as defined in \eqref{higher-f}, is an order $k$ nc function on $\(\bigsqcup_{m_0 = 1}^{\infty} \Omega^{(0)}_{s_0m_0} \)\times \hdots \times \(\bigsqcup_{m_k = 1}^{\infty} \Omega^{(k)}_{s_km_k} \)$ satisfying $f(Y^1,\ldots,Y^k)=g$. By Proposition 9.3 in \cite{Verb} and Lemma \ref{higher dse}, $f$ can be extended uniquely to an order $k$ nc function on 
$\Omega^{(0)}_{\rm d.s.e} \times \hdots \times \Omega^{(k)}_{\rm d.s.e.}.$
In particular, the restriction of the extended nc function to $\Omega^{(0)} \times \hdots \times \Omega^{(k)}$ is a nc function (of order $k$) as well.

Next, it will be shown that ${_j\Delta} f = F_j$ on $\(\bigsqcup_{m_0 = 1}^{\infty} \Omega^{(0)}_{s_0m_0} \)\times \hdots \times \(\bigsqcup_{m_k = 1}^{\infty} \Omega^{(k)}_{s_km_k}\)$. To do this, let $U \in \Omega^{(j)}_{s_jn_j}, V \in \Omega^{(j)}_{s_jq_j}, Z \in \R^{s_jn_j \times s_jq_j}$. We will show that 
\begin{alignat*}{2}
{_j\Delta} f&(X^0, \ldots, X^{j - 1}, U^j, V^j, X^{j + 1}, \ldots, X^k)(Z^1, \ldots, Z^j, Z, Z^{j + 1}, \ldots, Z^k) \\
&= F_j(X^0, \ldots, X^{j - 1}, U^j, V^j, X^{j + 1}, \ldots, X^k)(Z^1, \ldots, Z^j, Z, Z^{j + 1}, \ldots, Z^k). \end{alignat*}
Let $W = \fmat{cc} U & Z \\ 0 & V \emat \in \Omega^{(j)}_{s_jp_j}$ where $p_j = n_j + q_j$. Let $S \in \R^{s_jp_j \times s_jp_j}$ have the form
$S = \fmat{cc} I_{s_jn_j} & 0 \\ 0 & 0 \emat.$
Then, when $j = 0$,
\begin{multline*}
\fmat{cc} I_{s_jn_j} & 0 \\ 0 & 0 \emat f\(\fmat{cc} U & Z \\ 0 & V \emat, X^1, \ldots, X^k\)\(\fmat{c} 0 \\ Z^1 \emat, Z^2, \ldots, Z^k\) \\
 - f(I_{p_0} \otimes Y^0, X^1, \ldots, X^k)\(\fmat{cc} I_{s_jn_j} & 0 \\ 0 & 0 \emat \fmat{c} 0 \\ Z^1 \emat, Z^2, \ldots, Z^k\) 
\\
= \fmat{cc} I_{s_jn_j} & 0 \\ 0 & 0 \emat \fmat{c} f(U, X^1, \ldots, X^k)(0, Z^2, \ldots, Z^k) + {_0\Delta} f(U, V, X^1, \ldots, X^k)(Z, Z^1, Z^2, \ldots, Z^k) \\ f(V, X^1, \ldots, X^k)(Z^1, Z^2, \ldots, Z^k) \emat 
\end{multline*}
\begin{multline*}
 - \fmat{c} f(I_{n_0} \otimes Y^0, X^1, \ldots, X^k)\(0, Z^2, \ldots, Z^k\) \\ 0 \emat \\
= \fmat{c} {_0\Delta} f(U, V, X^1, \ldots, X^k)(Z, Z^1, \ldots, Z^k) \\ 0 \emat. \end{multline*}
At the same time,
\begin{alignat*}{2}
&\fmat{cc} I_{s_jn_j} & 0 \\ 0 & 0 \emat \fmat{cc} U & Z \\ 0 & V \emat - \fmat{cc} I_{n_j} \otimes Y^j & 0 \\ 0 & I_{q_j} \otimes Y^j \emat \fmat{cc} I_{s_jn_j} & 0 \\ 0 & 0 \emat 
= \fmat{cc} U - I_{n_j} \otimes Y^j & Z \\ 0 & 0 \emat. \end{alignat*}
Plugging this into the function $F_0$ gives
\begin{alignat*}{2}
F_0&\(I_{p_0} \otimes Y^0, \fmat{cc} U & Z \\ 0 & V \emat, X^1, \ldots, X^k\)\(\fmat{cc} U - I_{n_j} \otimes Y^j & Z \\ 0 & 0 \emat, \fmat{c} 0 \\ Z^1 \emat, Z^2, \ldots, Z^k\) \\
&= F_0\(I_{p_0} \otimes Y^0, U, X^1, \ldots, X^k\)\(\fmat{c} U - I_{n_j} \otimes Y^j \\ 0 \emat, 0, Z^2, \ldots, Z^k\) \\
&\hspace{5mm} + {_1\Delta} F_0(I_{p_0} \otimes Y^0, U, V, X^1, \ldots, X^k)\(\fmat{c} U - I_{n_j} \otimes Y^j \\ 0 \emat, Z, Z^1, \ldots, Z^k\) \\
&\hspace{5mm} + F_0\(I_{p_0} \otimes Y^0, V, X^1, \ldots, X^k\)\(\fmat{c} Z \\ 0 \emat, Z^1, \ldots, Z^k\) \\
&= \fmat{c} {_1\Delta} F_0(I_{n_0} \otimes Y^0, U, V, X^1, \ldots, X^k)\(U - I_{n_j} \otimes Y^j, Z, Z^1, \ldots, Z^k\) \\ 0 \emat \\
&\hspace{5mm} + \fmat{c} F_0\(I_{n_0} \otimes Y^0, V, X^1, \ldots, X^k\)\(Z, Z^1, \ldots, Z^k\) \\ 0 \emat\\
&\hspace{15mm} =\fmat{c} F_0(U, V, X^1, \ldots, X^k)(Z, Z^1, \ldots, Z^k) \\ 0 \emat. 
 \end{alignat*}
Thus
\begin{alignat*}{2}
S&f\(W, X^1, \ldots, X^k\)\(\fmat{c} 0 \\ Z^1 \emat, Z^2, \ldots, Z^k\) - f(I_{p_0} \otimes Y^0, X^1, \ldots, X^k)\(S\fmat{c} 0 \\ Z^1 \emat, Z^2, \ldots, Z^k\) \\
&\hspace{30mm}= F_0(I_{p_0} \otimes Y^0, W, X^1, \ldots, X^k)\(SW - (I_{p_j} \otimes Y^j)S, \fmat{c} 0 \\ Z^1 \emat, Z^2, \ldots, Z^k\) 
\end{alignat*}
can be written as
\begin{alignat*}{2}
&\fmat{c} {_0\Delta} f(U, V, X^1, \ldots, X^k)(Z, Z^1, \ldots, Z^k) \\ 0 \emat = \fmat{c} F_0(U, V, X^1, \ldots, X^k)(Z, Z^1, \ldots, Z^k) \\ 0 \emat. 
\end{alignat*}
Focusing on the top entries gives
\begin{alignat*}{2}
{_0\Delta} f&(U, V, X^1, \ldots, X^k)(Z, Z^1, \ldots, Z^k) = F_0(U, V, X^1, \ldots, X^k)(Z, Z^1, \ldots, Z^k). \end{alignat*}

When $j > 0$, the proof is similar. Finally, we know that $f, F_0, \ldots, F_k$ all have unique nc extensions to $\Omega^{(0)}_{d.s.e.} \times \hdots \times \Omega^{(k)}_{d.s.e.}$ by Proposition 9.3 in \cite{Verb}. It needs to be shown that on this larger set, ${_j\Delta} f = F_j$ for $j = 0, \ldots, k$. 

When $j = 0$, suppose by contradiction that for some $U \in \Omega^{(0)}_r, V \in \Omega^{(0)}_t$ and $Z \in \R^{r \times t}$, $${_0\Delta} f(U, V, X^1, \ldots, X^k)(Z, Z^1, \ldots, Z^k) \ne F(U, V, X^1, \ldots, X^k)(Z, Z^1, \ldots, Z^k).$$ For some integers $p_0$ and $q_0$, $U' = I_{p_0} \otimes U$ and $V' = I_{q_0} \otimes V$ are in $\Omega^{(0)}_{d.s.e.}$. Given $E$ as a $p_0 \times q_0$ matrix of all ones and $e$ as a $q_0 \times 1$ matrix of ones, we would have that for $X^j \in \Omega^{(j)}_{s_j}$ for $j = 1, \ldots, k$, $Z^1 \in \R^{t \times s_1}$ and $Z^j \in \R^{s_{j - 1} \times s_j}$ for $j = 2, \ldots, k$,
$${_0\Delta} f(U', V', X^1, \ldots, X^k)(E \otimes Z, e\otimes Z^1, \ldots, Z^k) \ne F(U', V', X^1, \ldots, X^k)(E \otimes Z, e\otimes Z^1, \ldots, Z^k).$$
But this is not the case. Hence ${_0\Delta} f = F_0$ on  $\Omega^{(0)}_{d.s.e.} \times \Omega_{s_1}^{(1)}\times \hdots \times \Omega_{s_k}^{(k)}$, and by extension, on $\Omega^{(0)}_{d.s.e.} \times \hdots \times \Omega^{(k)}_{d.s.e.}$. The proof for $j > 0$ is similar. Thus
${_j\Delta} f = F_j$ for $j = 0, \ldots, k$ on
$\Omega^{(0)}_{d.s.e.} \times \hdots \times \Omega^{(k)}_{d.s.e.}$. The result clearly holds on the smaller set $\Omega^{(0)} \times \hdots \times \Omega^{(k)}$. 
\end{proof}

The following corollary will be useful in later sections.

\begin{corollary}\label{nicer form}
Under the assumptions of Theorem \ref{almost result again}, the linear function $G(Z^1, \ldots, Z^k)$ can be written as
$$G(Z^1, \ldots, Z^k) = f(I_{m_0} \otimes Y^0, \ldots, I_{m_k} \otimes Y^k)(Z^1, \ldots, Z^k),$$
and the formula for the antiderivative $f$ can be written as
\begin{multline}\label{actual nicer form}
f(X^0, \ldots, X^k)(Z^1, \ldots, Z^k) 
= f(I_{m_0} \otimes Y^0, \ldots, I_{m_k} \otimes Y^k)(Z^1, \ldots, Z^k) \\
+ \sum_{j = 0}^k F_j(I_{m_0} \otimes Y^0, \ldots, I_{m_j} \otimes Y^j,X^j, \ldots, X^k)(Z^1, \ldots, Z^j, X^j - I_{m_j} \otimes Y^j, Z^{j + 1}, \ldots, Z^k)
\end{multline}
\end{corollary}

For the proof of Theorem \ref{higher-main}, it needs to be shown that a multilinear form (understood as a homomorphism on a tensor product) $g$ as in Theorem \ref{almost result again} always exists.
\begin{proposition}\label{bimodule}
Let $\N_0, \ldots, \N_k$ be $\R$-modules and $\mathcal{K} = \hom_R(\N_1^{s_0 \times s_1} \otimes \hdots \otimes \N_k^{s_{k - 1} \times s_k}, \N_0^{s_0 \times s_k})$. Then:
\begin{enumerate}
\item $\mathcal{K}$ is a bimodule over $\R^{s_0 \times s_0}$ with left and right actions defined by
\begin{alignat*}{2}
(S \cdot X)(Z^1, \ldots, Z^k) &= SX(Z^1, \ldots, Z^k), \\
(X \cdot S)(Z^1,  \ldots, Z^k) &= X(SZ^1, Z^2, \ldots, Z^k), \end{alignat*}
where $S \in \R^{s_0 \times s_0}, X \in \mathcal{K}$ and $Z^i \in \N_i^{s_{i - 1} \times s_i}$ for $i = 1, \ldots, k$. 
\item If $0 < j < k$, then $\mathcal{K}$ is a bimodule over $\R^{s_j \times s_j}$ with left and right actions defined by
\begin{alignat*}{2}
(S \cdot X)(Z^1,\ldots, Z^k) &= X(Z^1, \ldots, Z^{j - 1}, Z^jS, Z^{j + 1}, \ldots, Z^k) \\
(X \cdot S)(Z^1,\ldots,Z^k) &= X(Z^1, \ldots, Z^j, SZ^{j + 1}, Z^{j + 2}, \ldots, Z^k) \end{alignat*}
where $S \in \R^{s_j \times s_j}, X \in \mathcal{K}$ and $Z^i \in \N_i^{s_{i - 1} \times s_i}$ for $i = 1, \ldots, k$. 
\item $\mathcal{K}$ is a bimodule over $\R^{s_k \times s_k}$ with left and right actions defined by
\begin{alignat*}{2}
(S \cdot X)(Z^1, \ldots, Z^k) &= X(Z^1, \ldots, Z^{k - 1}, Z^kS) \\
(X \cdot S)(Z^1, \ldots, Z^k) &= X(Z^1, \ldots, Z^k)S \end{alignat*}
where $S \in \R^{s_k \times s_k}, X \in \mathcal{K}$ and $Z^i \in \N_i^{s_{i - 1} \times s_i}$ for $i = 1, \ldots, k$. 
\end{enumerate}
\end{proposition}

The proof is straightforward. With this proposition in mind, we make the following definitions.
\begin{definition} 
\begin{enumerate}
\item Define ${_0D}\colon \R^{s_0 \times s_0} \to \mathcal{K}$ by
\begin{alignat*}{2}
{_0D}(S) &= F_0(Y^0, Y^0, Y^1, \ldots, Y^k)(SY^0 - Y^0S), \end{alignat*}
where 
\begin{alignat*}{2}
& F_0(Y^0, Y^0, Y^1, \ldots, Y^k)(SY^0 - Y^0S)(Z^1, \ldots, Z^k) \\
&\hspace{30mm}= F_0(Y^0, Y^0, Y^1, \ldots, Y^k)(SY^0 - Y^0S, Z^1, \ldots, Z^k). \end{alignat*}
\item For $0< j < k$, define ${_jD}\colon \R^{s_j \times s_j} \to \mathcal{K}$ by
\begin{alignat*}{2}
{_jD}(S) &= F_j(Y^0, \ldots, Y^{j - 1}, Y^j, Y^j, Y^{j + 1}, \ldots, Y^k)(SY^j - Y^jS), \end{alignat*}
where 
\begin{alignat*}{2}
&F_j(Y^0, \ldots, Y^{j - 1}, Y^j, Y^j, Y^{j + 1}, \ldots, Y^k)(SY^j - Y^jS)(Z^1,\ldots, Z^k) \\
&\hspace{10mm}= F_j(Y^0, \ldots, Y^{j - 1}, Y^j, Y^j, Y^{j + 1}, \ldots, Y^k)(Z^1, \ldots, Z^j, SY^j - Y^jS, Z^{j + 1}, \ldots, Z^k). \end{alignat*}
\item Define ${_kD}: \R^{s_k \times s_k} \to \mathcal{K}$ by
\begin{alignat*}{2}
{_kD}(S) &= F_k(Y^0, \ldots, Y^{k - 1}, Y^k, Y^k)(SY^k - Y^kS), \end{alignat*}
where 
\begin{alignat*}{2}
& F_k(Y^0, \ldots, Y^{k - 1}, Y^k, Y^k)(SY^k - Y^kS)(Z^1, \ldots,  Z^k) \\
&\hspace{30mm}= F_k(Y^0, \ldots, Y^{k - 1}, Y^k, Y^k)(Z^1, \ldots, Z^k, SY^k - Y^kS). \end{alignat*}
\end{enumerate}
\end{definition}

\begin{proposition}\label{higher-derivations}
Let $\Omega^{(i)} \subseteq \M_{i, \nc}$ be right admissible nc sets, and let $Y^i \in \Omega^{(i)}_{s_i}$, $i=0,\ldots,k$. Suppose that, for some $j\in\{0,\ldots,k\}$, $$F_j \in \T^{k + 1}(\Omega^{(0)}, \ldots, \Omega^{(j - 1)}, \Omega^{(j)}, \Omega^{(j)}, \Omega^{(j + 1)}, \ldots, \Omega^{(k)}; \N_{0, \nc}, \ldots, \N_{j, \nc}, \M_{j, \nc}, \N_{j + 1, \nc}, \ldots, \N_{k, \nc}),$$ and 
$${_j\Delta} F_j(Y^0,\ldots,Y^{j-1},Y^j,Y^j,Y^j,Y^{j+1},\ldots,Y^k) = {_{j + 1}\Delta}F_j(Y^0,\ldots,Y^{j-1},Y^j,Y^j,Y^j,Y^{j+1},\ldots,Y^k).$$
 Then ${_jD}$ is a derivation on the algebra $R^{s_j \times s_j}$ with values in $\mathcal{K}$, and hence a Lie-algebra derivation.
\end{proposition}

The proof is essentially the same as the proof of Proposition \ref{derivationprop}, with just basic changes to adapt to the higher order nc functions involved.

\begin{proposition}\label{inner derivations}
In the assumptions of Proposition \ref{higher-derivations}, if $A, B, C \in \R^{s_j \times s_j}$ satisfy
$$AB = \lambda A \text{ and } BC = \lambda C \text{ for some } \lambda \in \R,$$
then
$$A\cdot{_jD}(B)\cdot C =0.$$
 In particular, if $A=B=C=P$ and $P^2=P$, then
$$P\cdot{_jD}(P)\cdot P =0.$$
\end{proposition}

The proof  is essentially the same as that of Proposition \ref{makingzero}, with minor changes for the higher order functions involved.

\begin{corollary}\label{inner derivation again}
In the assumptions of Proposition \ref{higher-derivations},
 ${_jD}$ is an inner derivation. 
\end{corollary}

This follows immediately from Theorem \ref{derivation properties} and Proposition \ref{inner derivations}.

Thus, we know that each of order $k + 1$ nc functions $F_0, \ldots, F_k$ can be written as the commutator of some order $k$ nc function, $f_0, \ldots, f_k$ respectively. It now needs to be shown that the functions $f_0, \ldots, f_k$ can be chosen equal to one another. We first prove the following proposition.

\begin{proposition}\label{with zero functions}
In the assumptions of Proposition \ref{higher-derivations}, the inner derivation ${_jD}$ can be defined by
${_jD}(S)=[S, g_j]$ with the $k$-linear map $g_j \in \mathcal{K}$ given by
\begin{equation}\label{gj}
g_j=  -\sum_{i=1}^{s_j} {_jD}(E_{i1})\cdot E_{1i}.
  \end{equation}
	 In other words, $g_j$ satisfies the corresponding $j$-th identity in  \eqref{defining formula j equals 0}--\eqref{defining formula j equals k} where $g$ is replaced by $g_j$, $j\in\{0,\ldots,k\}$.
\end{proposition}

\begin{proof}
In view of Proposition \ref{bimodule} and the subsequent definition, it suffices to show that
$$
E_{rs}\cdot g_j -g_j \cdot E_{rs}
={_jD}(E_{rs}),
$$
for all $r,s\in\R^{s_j\times s_j}$.
Using Proposition \ref{higher-derivations}, we obtain
\begin{multline*}
E_{rs}\cdot g_j -g_j \cdot E_{rs}
=-\sum_{i=1}^{s_j} E_{rs}\cdot {_jD}(E_{i1})\cdot E_{1i}
+\sum_{i=1}^{s_j}{_jD}(E_{i1}) \cdot (E_{1i}E_{rs})\\
=-\Big(\sum_{i=1}^{s_j}  {_jD}(E_{rs}E_{i1})\cdot E_{1i}
-\sum_{i=1}^{s_j}  {_jD}(E_{rs})\cdot (E_{i1}E_{1i})\Big)
+\sum_{i=1}^{s_j}{_jD}(E_{i1}) \cdot (E_{1i}E_{rs})\\
=-{_jD}(E_{r1})\cdot E_{1s}
+{_jD}(E_{rs})
+{_jD}(E_{r1}) \cdot E_{1s}
={_jD}(E_{rs}).
\end{multline*}
\end{proof}

\begin{proposition}\label{g}
Let $g_j \in\mathcal{K} $ be as in Proposition \ref{with zero functions}, $j=0,\ldots,k$, and suppose that ${_i\Delta F_j}={_{j+1}\Delta F_i}$, $0\le i\le j\le k$. Then $ g\in\mathcal{K}$ defined by
\begin{multline*}
g(Z^1,\ldots,Z^k)=g_0(Z^1,\ldots,Z^k)+\sum_{\ell=1}^k\sum_{i_0=1}^{s_0}\cdots\sum_{i_{\ell-1}=1}^{s_{\ell-1}}E_{i_0,1}g_{\ell}(E_{1,i_0}Z^1E_{i_1,1},\ldots,E_{1,i_{\ell-2}}Z^{\ell-1}E_{i_{\ell-1},1},\\
E_{1,i_{\ell-1}}Z^\ell,Z^{\ell+1},\ldots,Z^k)
\end{multline*}
satisfies \eqref{defining formula j equals 0}--\eqref{defining formula j equals k}.
\end{proposition}
For the proof of Proposition \ref{g}, we need the following lemma.
\begin{lemma}\label{lemma}
Let $c\in\mathcal{K}$, $j\in\{0,\ldots,k\}$. Then, for every $R_j\in\R^{s_j\times s_j}$,
\begin{equation}\label{commutator}
\Big[R_j,\sum_{i=1}^{s_j}E_{i1}\cdot c\cdot E_{1i}\Big]=0,
\end{equation}
where $\mathcal{K}=\hom_R(\N_1^{s_0 \times s_1} \otimes \hdots \otimes \N_k^{s_{k - 1} \times s_k}, \N_0^{s_0 \times s_k})$ is viewed as a bimodule over $\R^{s_j\times s_j}$.
\end{lemma}
\begin{proof}
Clearly, by linearity it suffices to show that \eqref{commutator} holds for $R_j=E_{rs}$ with arbitrary $r,s\in\{1,\ldots,s_j\}$. We have
\begin{equation*}
\Big[E_{rs},\sum_{i=1}^{s_j}E_{i1}\cdot c\cdot E_{1i}\Big]=\sum_{i=1}^{s_j}(E_{rs}E_{i1})\cdot c\cdot E_{1i}
-\sum_{i=1}^{s_j}E_{i1}\cdot c\cdot (E_{1i}E_{rs})
=E_{r1}\cdot c\cdot E_{1s}-E_{r1}\cdot c\cdot E_{1s}=0.
\end{equation*}
\end{proof}
\begin{proof}[Proof of Proposition \ref{g}]
For $j=0$, $g$ satisfies \eqref{defining formula j equals 0} by Proposition \ref{with zero functions} and Lemma \ref{lemma}.

For $0<j<k$, the right-hand side of \eqref{defining formula j greater 0} can be written as
\begin{multline*}
(R_j\cdot g-g\cdot R_j)(Z^1,\ldots,Z^k)=(R_j\cdot g_0-g_0\cdot R_j)(Z^1,\ldots,Z^k)\\
+\sum_{\ell=1}^{j-1}\sum_{i_0=1}^{s_0}\cdots\sum_{i_{\ell-1}=1}^{s_{\ell-1}}E_{i_0,1}g_\ell(E_{1,i_0}Z^1E_{i_1,1},\ldots,E_{1,i_{\ell-2}}Z^{\ell-1}E_{i_{\ell-1},1},E_{1,i_{\ell-1}}Z^\ell,Z^{\ell+1},\ldots,Z^jR_j,Z^{j+1},\ldots,Z^k)\\
-\sum_{\ell=1}^{j-1}\sum_{i_0=1}^{s_0}\cdots\sum_{i_{\ell-1}=1}^{s_{\ell-1}}E_{i_0,1}g_\ell(E_{1,i_0}Z^1E_{i_1,1},\ldots,E_{1,i_{\ell-2}}Z^{\ell-1}E_{i_{\ell-1},1},E_{1,i_{\ell-1}}Z^\ell,Z^{\ell+1},\ldots,Z^j,R_jZ^{j+1},\ldots,Z^k)\\
+\sum_{i_0=1}^{s_0}\cdots\sum_{i_{j-1}=1}^{s_{j-1}}E_{i_0,1}g_j(E_{1,i_0}Z^1E_{i_1,1},\ldots,E_{1,i_{j-2}}Z^{j-1}E_{i_{j-1},1},E_{1,i_{j-1}}Z^jR_j,Z^{j+1},\ldots,Z^k)\\
-\sum_{i_0=1}^{s_0}\cdots\sum_{i_{j-1}=1}^{s_{j-1}}E_{i_0,1}g_j(E_{1,i_0}Z^1E_{i_1,1},\ldots,E_{1,i_{j-2}}Z^{j-1}E_{i_{j-1},1},E_{1,i_{j-1}}Z^j,R_jZ^{j+1},Z^{j+2},\ldots,Z^k)\\
=-\sum_{p=1}^{s_0}F_0(Y^0,Y^0,Y^1,\ldots,Y^k)(E_{p,1}Y^0-Y^0E_{p,1},E_{1,p}Z^1,Z^2,\ldots,Z^{j-1},Z^jR_j,Z^{j+1},\ldots,Z^k)\\
+\sum_{p=1}^{s_0}F_0(Y^0,Y^0,Y^1,\ldots,Y^k)(E_{p,1}Y^0-Y^0E_{p,1},E_{1,p}Z^1,Z^2,\ldots,Z^{j-1},Z^j,R_jZ^{j+1},\ldots,Z^k)\\
-\sum_{\ell=1}^{j-1}\sum_{i_0=1}^{s_0}\cdots\sum_{i_{\ell-1}=1}^{s_{\ell-1}}\sum_{p=1}^{s_\ell}E_{i_0,1}F_\ell(Y^0,\ldots,Y^{\ell-1},Y^\ell,Y^\ell,Y^{\ell+1},\ldots,Y^k)
(E_{1,i_0}Z^1E_{i_1,1},\ldots,
E_{1,i_{\ell-2}}Z^{\ell-1}E_{i_{\ell-1},1},\\
E_{1,i_{\ell-1}}Z^\ell,
E_{p,1}Y^\ell-Y^\ell E_{p,1},E_{1,p}Z^{\ell+1},
Z^{\ell+2},\ldots,Z^{j-1},Z^jR_j,Z^{j+1},\ldots,Z^k)\\
+\sum_{\ell=1}^{j-1}\sum_{i_0=1}^{s_0}\cdots\sum_{i_{\ell-1}=1}^{s_{\ell-1}}\sum_{p=1}^{s_\ell}E_{i_0,1}F_\ell(Y^0,\ldots,Y^{\ell-1},Y^\ell,Y^\ell,Y^{\ell+1},\ldots,Y^k)(E_{1,i_0}Z^1E_{i_1,1},\ldots,E_{1,i_{\ell-2}}Z^{\ell-1}E_{i_{\ell-1},1},\\
E_{1,i_{\ell-1}}Z^\ell,E_{p,1}Y^\ell-Y^\ell E_{p,1},E_{1,p}Z^{\ell+1},
Z^{\ell+2},\ldots,Z^{j-1},Z^j,R_jZ^{j+1},\ldots,Z^k)\\
+\sum_{i_0=1}^{s_0}\cdots\sum_{i_{j-1}=1}^{s_{j-1}}E_{i_0,1}F_j(Y^0,\ldots,Y^{j-1},Y^j,Y^j,Y^{j+1},\ldots,Y^k)(E_{1,i_0}Z^1E_{i_1,1},\ldots,E_{1,i_{j-2}}Z^{j-1}E_{i_{j-1},1},\\
E_{1,i_{j-1}}Z^j,R_jY^j-Y^j R_j,Z^{j+1},\ldots,Z^k)\\    
=-\sum_{p=1}^{s_0}{_{j+1}\Delta}F_0(Y^0,Y^0,Y^1,\ldots,Y^{j-1},Y^j,Y^j,Y^{j+1},\ldots,Y^k)(E_{p,1}Y^0-Y^0E_{p,1},E_{1,p}Z^1,Z^2,\ldots,Z^j,\\
R_jY^j-Y^jR_j,Z^{j+1},\ldots,Z^k)\\
-\sum_{\ell=1}^{j-1}\sum_{i_0=1}^{s_0}\cdots\sum_{i_{\ell-1}=1}^{s_{\ell-1}}\sum_{p=1}^{s_\ell}E_{i_0,1}\,{_{j+1}\Delta}F_\ell(Y^0,\ldots,Y^{\ell-1},Y^\ell,Y^\ell,Y^{\ell+1},\ldots,Y^{j-1},Y^j,Y^j,Y^{j+1},\ldots,Y^k)\\
(E_{1,i_0}Z^1E_{i_1,1},\ldots,E_{1,i_{\ell-2}}Z^{\ell-1}E_{i_{\ell-1},1},E_{1,i_{\ell-1}}Z^\ell,E_{p,1}Y^\ell-Y^\ell E_{p,1},
E_{1,p}Z^{\ell+1},
Z^{\ell+2},\ldots,Z^j,\\
R_jY^j-Y^jR_j,Z^{j+1},\ldots,Z^k)\\
+\sum_{i_0=1}^{s_0}\cdots\sum_{i_{j-1}=1}^{s_{j-1}}E_{i_0,1}F_j(Y^0,\ldots,Y^{j-1},Y^j,Y^j,Y^{j+1},\ldots,Y^k)(E_{1,i_0}Z^1E_{i_1,1},\ldots,E_{1,i_{j-2}}Z^{j-1}E_{i_{j-1},1},\\
E_{1,i_{j-1}}Z^j,R_jY^j-Y^j R_j,Z^{j+1},\ldots,Z^k)\\
=-\sum_{p=1}^{s_0}{_0\Delta}F_j(Y^0,Y^0,Y^1,\ldots,Y^{j-1},Y^j,Y^j,Y^{j+1},\ldots,Y^k)(E_{p,1}Y^0-Y^0E_{p,1},E_{1,p}Z^1,Z^2,\ldots,Z^j,\\
R_jY^j-Y^jR_j,Z^{j+1},\ldots,Z^k)
\\
\end{multline*}
\begin{multline*}
-\sum_{\ell=1}^{j-1}\sum_{i_0=1}^{s_0}\cdots\sum_{i_{\ell-1}=1}^{s_{\ell-1}}\sum_{p=1}^{s_\ell}E_{i_0,1}\,{_\ell \Delta}F_j(Y^0,\ldots,Y^{\ell-1},Y^\ell,Y^\ell,Y^{\ell+1},\ldots,Y^{j-1},Y^j,Y^j,Y^{j+1},\ldots,Y^k)\\
(E_{1,i_0}Z^1E_{i_1,1},\ldots,E_{1,i_{\ell-2}}Z^{\ell-1}E_{i_{\ell-1},1},E_{1,i_{\ell-1}}Z^\ell,E_{p,1}Y^\ell-Y^\ell E_{p,1},
E_{1,p}Z^{\ell+1},
Z^{\ell+2},\ldots,Z^j,\\
R_jY^j-Y^jR_j,Z^{j+1},\ldots,Z^k)\\
+\sum_{i_0=1}^{s_0}\cdots\sum_{i_{j-1}=1}^{s_{j-1}}E_{i_0,1}F_j(Y^0,\ldots,Y^{j-1},Y^j,Y^j,Y^{j+1},\ldots,Y^k)(E_{1,i_0}Z^1E_{i_1,1},\ldots,E_{1,i_{j-2}}Z^{j-1}E_{i_{j-1},1},\\
E_{1,i_{j-1}}Z^j,R_jY^j-Y^j R_j,Z^{j+1},\ldots,Z^k)\\
=-\sum_{p=1}^{s_0}E_{p,1}F_j(Y^0,\ldots,Y^{j-1},Y^j,Y^j,Y^{j+1},\ldots,Y^k)(E_{1,p}Z^1,Z^2,\ldots,Z^j,
R_jY^j-Y^jR_j,Z^{j+1},\ldots,Z^k)\\
+ \sum_{p=1}^{s_0}F_j(Y^0,\ldots,Y^{j-1},Y^j,Y^j,Y^{j+1},\ldots,Y^k)(E_{p,p}Z^1,Z^2,\ldots,Z^j,
R_jY^j-Y^jR_j,Z^{j+1},\ldots,Z^k)\\
-\sum_{\ell=1}^{j-1}\sum_{i_0=1}^{s_0}\cdots\sum_{i_{\ell-1}=1}^{s_{\ell-1}}\sum_{p=1}^{s_\ell}E_{i_0,1}F_j(Y^0,\ldots,Y^{j-1},Y^j,Y^j,Y^{j+1},\ldots,Y^k)(E_{1,i_0}Z^1E_{i_1,1},\ldots,
E_{1,i_{\ell-2}}Z^{\ell-1}E_{i_{\ell-1},1},\\
E_{1,i_{\ell-1}}Z^\ell E_{p,1},
E_{1,p}Z^{\ell+1},
Z^{\ell+2},\ldots,Z^j,R_jY^j-Y^jR_j,Z^{j+1},\ldots,Z^k)\\
+ \sum_{\ell=1}^{j-1}\sum_{i_0=1}^{s_0}\cdots\sum_{i_{\ell-1}=1}^{s_{\ell-1}}\sum_{p=1}^{s_\ell}E_{i_0,1}F_j(Y^0,\ldots,Y^{j-1},Y^j,Y^j,Y^{j+1},\ldots,Y^k)(E_{1,i_0}Z^1E_{i_1,1},\ldots,
E_{1,i_{\ell-2}}Z^{\ell-1}E_{i_{\ell-1},1},\\
E_{1,i_{\ell-1}}Z^\ell,
E_{p,p}Z^{\ell+1},
Z^{\ell+2},\ldots,Z^j,R_jY^j-Y^jR_j,Z^{j+1},\ldots,Z^k)\\
+\sum_{i_0=1}^{s_0}\cdots\sum_{i_{j-1}=1}^{s_{j-1}}E_{i_0,1}F_j(Y^0,\ldots,Y^{j-1},Y^j,Y^j,Y^{j+1},\ldots,Y^k)(E_{1,i_0}Z^1E_{i_1,1},\ldots,E_{1,i_{j-2}}Z^{j-1}E_{i_{j-1},1},\\
E_{1,i_{j-1}}Z^j,R_jY^j-Y^j R_j,Z^{j+1},\ldots,Z^k)\\    
=-\sum_{p=1}^{s_0}E_{p,1}F_j(Y^0,\ldots,Y^{j-1},Y^j,Y^j,Y^{j+1},\ldots,Y^k)(E_{1,p}Z^1,Z^2,\ldots,Z^j,R_jY^j-Y^jR_j,Z^{j+1},\ldots,Z^k)\\
+ F_j(Y^0,\ldots,Y^{j-1},Y^j,Y^j,Y^{j+1},\ldots,Y^k)(Z^1,\ldots,Z^j,R_jY^j-Y^jR_j,Z^{j+1},\ldots,Z^k)\\
-\sum_{\ell=1}^{j-1}\sum_{i_0=1}^{s_0}\cdots\sum_{i_{\ell-1}=1}^{s_{\ell-1}}\sum_{p=1}^{s_\ell}E_{i_0,1}F_j(Y^0,\ldots,Y^{j-1},Y^j,Y^j,Y^{j+1},\ldots,Y^k)(E_{1,i_0}Z^1E_{i_1,1},\ldots,
E_{1,i_{\ell-2}}Z^{\ell-1}E_{i_{\ell-1},1},\\
E_{1,i_{\ell-1}}Z^\ell E_{p,1},
E_{1,p}Z^{\ell+1},
Z^{\ell+2},\ldots,Z^j,R_jY^j-Y^jR_j,Z^{j+1},\ldots,Z^k)\\
+ \sum_{\ell=1}^{j-1}\sum_{i_0=1}^{s_0}\cdots\sum_{i_{\ell-1}=1}^{s_{\ell-1}}E_{i_0,1}F_j(Y^0,\ldots,Y^{j-1},Y^j,Y^j,Y^{j+1},\ldots,Y^k)(E_{1,i_0}Z^1E_{i_1,1},\ldots,
E_{1,i_{\ell-2}}Z^{\ell-1}E_{i_{\ell-1},1},\\
E_{1,i_{\ell-1}}Z^\ell,
Z^{\ell+1},
\ldots,Z^j,R_jY^j-Y^jR_j,Z^{j+1},\ldots,Z^k)\\
+\sum_{i_0=1}^{s_0}\cdots\sum_{i_{j-1}=1}^{s_{j-1}}E_{i_0,1}F_j(Y^0,\ldots,Y^{j-1},Y^j,Y^j,Y^{j+1},\ldots,Y^k)(E_{1,i_0}Z^1E_{i_1,1},\ldots,E_{1,i_{j-2}}Z^{j-1}E_{i_{j-1},1},\\
E_{1,i_{j-1}}Z^j,R_jY^j-Y^j R_j,Z^{j+1},\ldots,Z^k)\\ 
=F_j(Y^0,\ldots,Y^{j-1},Y^j,Y^j,Y^{j+1},\ldots,Y^k)(Z^1,\ldots,Z^j,R_jY^j-Y^jR_j,Z^{j+1},\ldots,Z^k),   
 \end{multline*}
which is the left-hand side of \eqref{defining formula j greater 0}.

For $j=k$, the proof is analogous, with the modifications corresponding to the definition of the bimodule $\mathcal{K}$ in this case.
\end{proof}
 
We are now in a position to prove the main theorem of this section.
\begin{proof}[Proof of Theorem \ref{higher-main}]
Suppose  there exists an $f \in \T^k(\Omega^{(0)}, \ldots, \Omega^{(k)}; \N_{0, \nc}, \ldots, \N_{k, \nc})$ such that ${_j\Delta} f = F_j,\quad j = 0, \ldots, k,$. Then the equalities ${_i\Delta} F_j = {_{j + 1}\Delta} F_i$ follow by Proposition \ref{deltas}.

Conversely, if the equaities ${_i\Delta} F_j = {_{j + 1}\Delta} F_i$ hold for all $i,j$ such that $0\le i\le j\le k$, then by Proposition \ref{g} there exists a $g\in\mathcal{K}$ satisfying \eqref{defining formula j equals 0}--\eqref{defining formula j equals k}. This, in turn, by Theorem \ref{almost result again} guarantees the existence of $f \in \T^k(\Omega^{(0)}, \ldots, \Omega^{(k)}; \N_{0, \nc}, \ldots, \N_{k, \nc})$ such that ${_j\Delta} f = F_j$, $j = 0, \ldots, k$, and moreover, \eqref{higher-f} holds.
\end{proof}

\section{Special Cases of Integrability}

In this section, we will look at the major theorem for three specific subsets of nc functions. In each case, we will look at how the special features of each set effect the main theorem.

\subsection{The Modules $\M$ are of the Form $\M = \R^d$}

Suppose that the modules $\M_0$, \ldots, $\M_k$ under consideration have the special form $\R^{d_0}$, \ldots, $\R^{d_k}$. In this case, a directional difference-differential operator can be defined in position $j$ as follows:

For $j = 0$:
\begin{alignat*}{2}
{_0\Delta}_{\alpha}f(X^0_1, X^0_2, X^1 \ldots, X^k)&: \R^{n_0^1 \times n_0^2} \times \N_1^{n_0^2 \times n_1} \times \hdots \times \N_k^{n_{k - 1} \times n_k} \to \N_0^{n_0^1 \times n_k} \\
{_0\Delta}_{\alpha}f(X^0_1, X^0_2, X^1 \ldots, X^k)&(A, Z^1, \ldots, Z^k) = {_0\Delta}f(X^0_1, X^0_2, X^1, \ldots, X^k)(Ae_{\alpha}, Z^1, \ldots, Z^k). \end{alignat*}

For $0 < j < k$:

\begin{alignat*}{2}
{_j\Delta}_{\alpha}&f(X^0, \ldots, X^{j - 1}, X^j_1, X^j_2, \ldots, X^k): \\
&\hspace{25mm} \N_1^{n_0 \times n_1} \times \hdots \times \N_j^{n_{j - 1} \times n_j^1} \times \R^{n_j^1 \times n_j^2} \times \N_{j + 1}^{n_j^2 \times n_{j + 1}} \times \hdots \times \N_k^{n_{k - 1} \times n_k} \to \N_0^{n_0 \times n_k} \\
&{_j\Delta}_{\alpha}f(X^0, \ldots, X^{j - 1}, X^j_1, X^j_2, \ldots, X^k)(Z^1, \ldots, Z^j, A, Z^{j + 1}, \ldots, Z^k) \\
&\hspace{25mm} = {_j\Delta}f(X^0, \ldots, X^{j - 1}, X^j_1, X^j_2, \ldots, X^k)(Z^1, \ldots, Z^j, Ae_{\alpha}, Z^{j + 1}, \ldots, Z^k). \end{alignat*}

For $j = k$:

\begin{alignat*}{2}
{_k\Delta}_{\alpha}f(X^0, \ldots, X^{k - 1}, X^k_1, X^k_2)&: \N_1^{n_0 \times n_1} \times \hdots \times \N_k^{n_{k - 1} \times n_k^1} \times \R^{n_k^1 \times n_k^2} \to \N_0^{n_0 \times n_k^2} \\
{_k\Delta}_{\alpha}f(X^0, \ldots, X^{k - 1}, X^k_1, X^k_2)&(Z^1, \ldots, Z^k, A) = {_k\Delta}f(X^0, \ldots, X^{k - 1}, X^k_1, X^k_2)(Z^1, \ldots, Z^k, Ae_{\alpha}). \end{alignat*}

In each of these cases, we can rewrite the general difference-differential operators in terms of the directional difference-differential operators as follows:

For $j = 0$:

\begin{alignat*}{2}
{_0\Delta}f(X^0_1, X^0_2, X^1 \ldots, X^k)(Z, Z^1, \ldots, Z^k) = \sum_{\alpha = 1}^{d_0} {_0\Delta}_{\alpha}f(X^0_1, X^0_2, X^1 \ldots, X^k)(Z_{\alpha}, Z^1, \ldots, Z^k). \end{alignat*}

For $0 < j < k$:

\begin{alignat*}{2}
{_j\Delta}&f(X^0, \ldots, X^{j - 1}, X^j_1, X^j_2, \ldots, X^k)(Z^1, \ldots, Z^j, Z, Z^{j + 1}, \ldots, Z^k) \\
&= \sum_{\alpha = 1}^{d_j} {_j\Delta}_{\alpha}f(X^0, \ldots, X^{j - 1}, X^j_1, X^j_2, \ldots, X^k)(Z^1, \ldots, Z^j, Z_{\alpha}, Z^{j + 1}, \ldots, Z^k). \end{alignat*}

For $j = k$:

$${_k\Delta}f(X^0, \ldots, X^{k - 1}, X^k_1, X^k_2)(Z^1, \ldots, Z^k, Z) = \sum_{\alpha = 1}^{d_k} {_k\Delta}_{\alpha}f(X^0, \ldots, X^{k - 1}, X^k_1, X^k_2)(Z^1, \ldots, Z^k, Z_{\alpha}).$$

Now, suppose we have $k + 1$ nc functions $F_0, \ldots, F_k$ each of which is order $k + 1$. Since $\M_j = \R^{d_j}$, we can use linearity to rewrite each function as follows:

\begin{alignat*}{2}
F_0(X^0_1, X^0_2, X^1 \ldots, X^k)(Z, Z^1, \ldots, Z^k) &= \sum_{\beta = 1}^{d_0} F_{0, \beta}(X^0_1, X^0_2, X^1 \ldots, X^k)(Z_{\beta}, Z^1, \ldots, Z^k), \\
&\hspace{-70mm}F_j(X^0, \ldots, X^{j - 1}, X^j_1, X^j_2, X^{j + 1}, \ldots, X^k)(Z^1, \ldots, Z^j, Z, Z^{j + 1}, \ldots, Z^k) \\
&\hspace{-50mm} = \sum_{\beta = 1}^{d_j} F_{j, \beta}(X^0, \ldots, X^{j - 1}, X^j_1, X^j_2, X^{j + 1}, \ldots, X^k)(Z^1, \ldots, Z^j, Z_{\beta}, Z^{j + 1}, \ldots, Z^k), \\
F_k(X^0, \ldots, X^{k - 1}, X^k_1, X^k_2)(Z^1, \ldots, Z^k, Z) &= \sum_{\beta = 1}^{d_k} F_{k,\beta}f(X^0, \ldots, X^{k - 1}, X^k_1, X^k_2)(Z^1, \ldots, Z^k, Z_{\beta}). \end{alignat*}

With this background established, we can write the main theorem as follows:
\begin{theorem}
Let $\Omega^{(0)} \subseteq \R^{d_0}, \ldots, \Omega^{(k)} \subseteq \R^{d_k}$ be right admissible nc sets. Let $$F_{j,\beta} \in \T^{k + 1}(\Omega^{(0)}, \ldots, \Omega^{(j - 1)}, \Omega^{(j)}, \Omega^{(j)}, \Omega^{(j + 1)}, \ldots, \Omega^{(k)}; \N_{0, \nc}, \ldots, \N_{j, \nc}, \R_{\nc}, \N_{j + 1, \nc}, \ldots, \N_{k, \nc})$$ for $j = 0, \ldots, k$ and $\beta=1,\ldots,d_j$. Then there exists an $f \in \T^k(\Omega^{(0)}, \ldots, \Omega^{(k)}; \N_{0, \nc}, \ldots, \N_{k, \nc})$ such that ${_j\Delta}_{\beta}f = F_{j, \beta}$ for $j = 0, \ldots, k$ and $\beta=1,\ldots,d_j$ if and only if ${_i\Delta}_\alpha F_{j, \beta} = {_{j + 1}\Delta}_{\beta} F_{i, \alpha}$ for $0 \leq i \leq j \leq k$, $\alpha=1,\ldots,d_i$, and $\beta=1,\ldots,d_j$.
Furthermore, $f$ is uniquely defined up to a $k$-linear mapping $c\colon \N_1\times \cdots\times \N_k \to \N_0$, i.e.,  if $\tilde{f}$ is another antiderivative, then
$$\tilde{f}(X^0, \ldots, X^k)(Z^1, \ldots, Z^k) = f(X^0, \ldots, X^k)(Z^1, \ldots, Z^k) + C(Z^1, \ldots, Z^k),$$
with $C$ defined as in $\eqref{C}$.
\end{theorem}

\begin{proof}
It only needs to be shown that the conditions ${_i\Delta}_\alpha F_{j, \beta} = {_{j + 1}\Delta}_{\beta} F_{i, \alpha}$ for all $\alpha,\beta$ are in this case equivalent to the condition ${_i\Delta}F_j = {_{j + 1}\Delta}F_i$, where $F_j$'s are defined earlier in this section.

First, suppose that ${_i\Delta}_\alpha F_{j, \beta} = {_{j + 1}\Delta}_{\beta} F_{i, \alpha}$ for all appropriate $i,j,\alpha, \beta$. Then
\begin{alignat*}{2}
&{_i\Delta}F_j(X^0, \ldots, X^{i - 1}, X^i_1, X^i_2, X^{i + 1}, \ldots, X^{j - 1}, X^j_1, X^j_2, X^{j + 1}, \ldots, X^k) \\
&\hspace{50mm} (Z^1, \ldots, Z^i, Z', Z^{i + 1}, \ldots, Z^j, Z'', Z^{j + 1}, \ldots, Z^k) \\
&=\sum_{\alpha = 1}^{d_i} \sum_{\beta = 1}^{d_j} {_i\Delta}_{\alpha}F_{j, \beta}(X^0, \ldots, X^{i - 1}, X^i_1, X^i_2, X^{i + 1}, \ldots, X^{j - 1}, X^j_1, X^j_2, X^{j + 1}, \ldots, X^k) \\
&\hspace{50mm} (Z^1, \ldots, Z^i, Z'_{\alpha}, Z^{i + 1}, \ldots, Z^j, Z''_{\beta}, Z^{j + 1}, \ldots, Z^k) \\
&\hspace{15mm}= \sum_{\alpha = 1}^{d_i} \sum_{\beta = 1}^{d_j} {_{j + 1}\Delta}_{\beta} F_{i, \alpha}(X^0, \ldots, X^{i - 1}, X^i_1, X^i_2, X^{i + 1}, \ldots, X^{j - 1}, X^j_1, X^j_2, X^{j + 1}, \ldots, X^k) \\
&\hspace{65mm} (Z^1, \ldots, Z^i, Z'_{\alpha}, Z^{i + 1}, \ldots, Z^j, Z''_{\beta}, Z^{j + 1}, \ldots, Z^k) \\
&\hspace{15mm}= {_{j + 1}\Delta}F_i(X^0, \ldots, X^{i - 1}, X^i_1, X^i_2, X^{i + 1}, \ldots, X^{j - 1}, X^j_1, X^j_2, X^{j + 1}, \ldots, X^k) \\
&\hspace{65mm} (Z^1, \ldots, Z^i, Z', Z^{i + 1}, \ldots, Z^j, Z'', Z^{j + 1}, \ldots, Z^k). \end{alignat*}

Conversely, suppose that ${_i\Delta}F_j = {_{j + 1}\Delta}F_i$. Then
\begin{alignat*}{2}
&{_i\Delta}_{\alpha}F_{j, \beta}(X^0, \ldots, X^{i - 1}, X^i_1, X^i_2, X^{i + 1}, \ldots, X^{j - 1}, X^j_1, X^j_2, X^{j + 1}, \ldots, X^k) \\
&\hspace{50mm} (Z^1, \ldots, Z^i, Z'_{\alpha}, Z^{i + 1}, \ldots, Z^j, Z''_{\beta}, Z^{j + 1}, \ldots, Z^k) \\
&= {_i\Delta}F_j(X^0, \ldots, X^{i - 1}, X^i_1, X^i_2, X^{i + 1}, \ldots, X^{j - 1}, X^j_1, X^j_2, X^{j + 1}, \ldots, X^k) \\
&\hspace{50mm} (Z^1, \ldots, Z^i, Z'e_{\alpha}, Z^{i + 1}, \ldots, Z^j, Z''e_{\beta}, Z^{j + 1}, \ldots, Z^k) \\
&= {_{j + 1}\Delta}F_i(X^0, \ldots, X^{i - 1}, X^i_1, X^i_2, X^{i + 1}, \ldots, X^{j - 1}, X^j_1, X^j_2, X^{j + 1}, \ldots, X^k) \\
&\hspace{50mm} (Z^1, \ldots, Z^i, Z'e_{\alpha}, Z^{i + 1}, \ldots, Z^j, Z''e_{\beta}, Z^{j + 1}, \ldots, Z^k) \\
&= {_{j + 1}\Delta}_{\beta} F_{i, \alpha}(X^0, \ldots, X^{i - 1}, X^i_1, X^i_2, X^{i + 1}, \ldots, X^{j - 1}, X^j_1, X^j_2, X^{j + 1}, \ldots, X^k) \\
&\hspace{65mm} (Z^1, \ldots, Z^i, Z'_{\alpha}, Z^{i + 1}, \ldots, Z^j, Z''_{\beta}, Z^{j + 1}, \ldots, Z^k).\end{alignat*}
\end{proof}

\subsection{The Functions are NC Polynomials}

We define a nc polynomial of order $k$ as follows. Let $x^j=(x^j_1,\ldots,x^j_{d_j})$, $j = 0, \ldots, k$, and $z^j=(z^j_1,\ldots,z^j_{d'_j})$, $j=1,\ldots,k$, be tuples of free noncommuting indeterminates.
Let $\mathcal{G}_{d_j}$ and $\mathcal{G}_{d'_{j}}$ be free monoids on $d_j$ and $d'_{j}$ generators, $g^j_1,\ldots, g^j_{d_j}$ and $g'^j_1,\ldots,g'^j_{d'_j}$, respectively. Let $\mathcal{G} = \prod_{j = 0}^k \mathcal{G}_{d_j}$ and $\mathcal{G}' = \prod_{j = 1}^k \mathcal{G}_{d'_{j}}$. Let $\N$ be a module over a noncommutative unital ring $\R$, $p_{(w, v)} \in \N$ for $w \in \mathcal{G}$ and $v \in \mathcal{G}'$.  Then a nc polynomial of order $k$ is given by $$p = \sum_{\underset{|w| \leq L}{w \in\mathcal{G}}} \sum_{\underset{|v^1| = \hdots = |v^k| = 1}{v \in \mathcal{G}'}} p_{(w, v)} (x^0)^{w_0}(z^1)^{v_1}(x^1)^{w_1} \hdots (z^k)^{v_k}(x^k)^{w_k}.$$
Here $|w| = |w^0| + \hdots + |w^k|$, where $|w^j|$ is the length of the $j$th word. Notice that the module of nc polynomials of order $k$ over $\R$ can be naturally identified with 
$$\bigoplus_{\underset{|v^1| = \hdots = |v^k| = 1}{v \in \mathcal{G}'}}\R\langle x^1\rangle\otimes\cdots\otimes\R\langle x^k\rangle,$$
where $\R\langle x^j\rangle$ denotes the module of nc polynomials (of order 0) with coefficients in $\R$.

One can evaluate $p$ on matrices as follows. Let
$X^j \in \(\R^{n_j \times n_j}\)^{d_j}$, $j = 0, \ldots, k$, and let $Z^j \in \(\R^{n_{j - 1} \times n_j}\)^{d'_j}$, $j = 1, \ldots, k$.  Then $$p(X^0, \ldots, X^k)(Z^1, \ldots, Z^k) = \sum_{\underset{|w| \leq L}{w \in\mathcal{G}}} \sum_{\underset{|v^1| = \hdots = |v^k| = 1}{v \in \mathcal{G}'}} p_{(w, v)} (X^0)^{w_0}(Z^1)^{v_1}(X^1)^{w_1} \hdots (Z^k)^{v_k}(X^k)^{w_k} \in \N^{n_0 \times n_k},$$
and a nc polynomial  of order $k$ can be viewed as a nc function of order $k$: 
$$p \in \T^k(\R^{d_0}_{\nc}, \ldots, \R^{d_k}_{\nc}; \N_{\nc}, \R^{d'_{1}}_{\nc}, \ldots, 
\R^{d'_{k}}_{\nc}).$$ 

For a word $w \in \mathcal{G}_{d}$, let $w_{[i]}$ be the subword containing the first $i$ letters of $w$ and let $_{[i]}w$ be the subword containing the last $i$ letters of $w$.

\begin{proposition}\label{GHV}
A first order nc polynomial $p$ is integrable, that is, there exists a (zero-order) nc polynomial $q$ such that $\Delta q=p$, if and only if for every monomial of positive total degree, $p_wx^{w_{[i_0]}}z_{j_{i_0 + 1}}y^{_{[|w| - i_0 - 1]}w}$ that is in $p$, $p_w\sum_{i = 0}^{|w| - 1} x^{w_{[i]}}z_{j_{i + 1}}y^{_{[|w| - i - 1]}w}$ is also in $p$. In this case, $q$ has the form $q=q_\emptyset+\sum_{0<|w|\le L}p_wx^w$.
\end{proposition}

\begin{proof}
Suppose that $p$ has the desired form, i.e., can be written as
$$p = \sum_{w \in \mathcal{G}_d\colon|w|\le L} p_w\sum_{i = 0}^{|w|-1} x^{w_{[i]}}z_{j_{i+ 1}}y^{_{[|w| - i - 1]}w}.$$
Then, clearly, $q(x)=q_\emptyset+\sum_{0<|w|\le L}p_wx^w$ is an antiderivative of $p$. Conversely, given the zero-order polynomial $q$ as above, we can easily see that $p=\Delta q$ has the desired property.
\end{proof}

We note that Proposition \ref{GHV} has been proved, in an equivalent form, in \cite[Theorem 2.4]{Vin}.

\begin{remark}
One can show that the property of the first order polynomial $p$ stated in Proposition \ref{GHV} is equivalent to the condition that ${_0\Delta} p={_1\Delta} p$ directly, bypassing Theorem \ref{main}. Indeed,
suppose that $p$ has the desired form and hence can be written as
$$p = \sum_{w \in \mathcal{G}_d\colon|w|\le L} p_w\sum_{i = 0}^{|w|-1} x^{w_{[i]}}z_{j_{i+ 1}}y^{_{[|w| - i - 1]}w}.$$
Then
\begin{multline*}
p\left(\begin{bmatrix} X & Z^1 \\ 0 & W \end{bmatrix},Y\right)\left(\begin{bmatrix} Z \\ Z^2 \end{bmatrix}\right)=\sum_{w } p_w\sum_{i = 0}^{|w| - 1} \begin{bmatrix} X & Z^1 \\ 0 & W \end{bmatrix}^{w_{[i]}}\begin{bmatrix} Z_{j_{i + 1}} \\ Z_{j_{i + 1}}^2 \end{bmatrix}Y^{_{[|w| - i - 1]}w} \\
= \sum_{w } p_w\sum_{i = 0}^{|w| - 1} \begin{bmatrix} X^{w_{[i]}}Z_{j_{i + 1}}Y^{_{[|w| - i - 1]}w} + \sum_{\ell = 0}^{i - 1} X^{(w_{[i]})_{[\ell]}}Z^1_{j_{\ell + 1}}W^{_{[|w_{[i]}| - \ell - 1]}{(w_{[i]})}}Z^2_{j_{i + 1}}Y^{_{[|w| - i - 1]}w} \\ 
W^{w_{[i]}}Z^2_{j_{{i + 1}}}Y^{_{[|w| - i - 1]}w} \end{bmatrix}\\
= \sum_{w } p_w\sum_{i = 0}^{|w| - 1} \begin{bmatrix} X^{w_{[i]}}Z_{j_{i + 1}}Y^{_{[|w| - i - 1]}w} + \sum_{\ell = 0}^{i - 1} X^{w_{[\ell]}}Z^1_{j_{\ell + 1}}W^{_{[i - \ell - 1]}{(w_{[i]})}}Z^2_{j_{i + 1}}Y^{_{[|w| - i - 1]}w} \\ 
W^{w_{[i]}}Z^2_{j_{{i + 1}}}Y^{_{[|w| - i - 1]}w} \end{bmatrix}
 \end{multline*}
and
\begin{multline*}
p\left(X,\begin{bmatrix} W & Z^2 \\ 0 & Y \end{bmatrix}\right)\left( \begin{bmatrix} Z^1 & Z \end{bmatrix}\right)=\sum_{w } p_w\sum_{i = 0}^{|w| - 1} X^{w_{[i]}}\begin{bmatrix} Z^1_{j_{i+1}} & Z_{j_{i+1}} \end{bmatrix}\begin{bmatrix} W & Z^2 \\ 0 & Y \end{bmatrix}^{_{[|w| - i - 1]}w} \\
= \sum_{w } p_w\sum_{i = 0}^{|w| - 1} \[ \vphantom{\sum_{\ell = 1}^{|w| - i - 1}} X^{w_{[i]}}Z^1_{j_{i+1}}W^{_{[|w| - i - 1]}w} \right. \\
 \left. \sum_{\ell = i + 1}^{|w| - 1} X^{w_{[i]}}Z^1_{j_{i + 1}}W^{({_{[|w| - i - 1]}w})_{[\ell - i - 1]}}Z^2_{j_{\ell + 1}}Y^{_{[|w| - \ell -1]}({_{[|w| - i - 1]}w})} + X^{w_{[i]}}Z_{j_{i + 1}}Y^{w_{[|w| - i - 1]}} \] \\
= \sum_{w } p_w\sum_{i = 0}^{|w| - 1} \[ \vphantom{\sum_{\ell = 1}^{|w| - i - 1}} X^{w_{[i]}}Z^1_{j_{i+1}}W^{_{[|w| - i - 1]}w} \right. \\
 \left. \sum_{\ell = i + 1}^{|w| - 1} X^{w_{[i]}}Z^1_{j_{i + 1}}W^{({_{[|w| - i - 1]}w})_{[\ell - i - 1]}}Z^2_{j_{\ell + 1}}Y^{_{[|w| - \ell - 1]}w} + X^{w_{[i]}}Z_{j_{i + 1}}Y^{w_{[|w| - i - 1]}} \]. \end{multline*}
Hence,
\begin{gather*}
\begin{aligned}
{_0\Delta}p(X, W, Y)(Z^1, Z^2) &=\sum_{w } p_w\sum_{i = 0}^{|w| - 1} \sum_{\ell = 0}^{i - 1} X^{w_{[\ell]}}Z^1_{j_{\ell + 1}}W^{_{[i - \ell - 1]}{(w_{[i]})}}Z^2_{j_{i + 1}}Y^{_{[|w| - i - 1]}w}\\
&=\sum_{w } p_w\sum_{i = 1}^{|w| - 1} \sum_{\ell = 0}^{i - 1} X^{w_{[\ell]}}Z^1_{j_{\ell + 1}}W^{_{[i - \ell - 1]}{(w_{[i]})}}Z^2_{j_{i + 1}}Y^{_{[|w| - i - 1]}w},
\end{aligned} \\
\begin{aligned}
{_1\Delta}p(X, W, Y)(Z^1, Z^2) &= \sum_{w } p_w\sum_{i = 0}^{|w| - 1} \sum_{\ell = i + 1}^{|w| - 1} X^{w_{[i]}}Z^1_{j_{i + 1}}W^{({_{[|w| - i - 1]}w})_{[\ell - i - 1]}}Z^2_{j_{\ell + 1}}Y^{_{[|w| - \ell - 1]}w}\\
&= \sum_{w } p_w\sum_{\ell = 1}^{|w| - 1} \sum_{i = 0}^{\ell - 1} X^{w_{[i]}}Z^1_{j_{i + 1}}W^{({_{[|w| - i - 1]}w})_{[\ell - i - 1]}}Z^2_{j_{\ell + 1}}Y^{_{[|w| - \ell - 1]}w}\\
&= \sum_{w } p_w\sum_{i = 1}^{|w| - 1} \sum_{\ell = 0}^{i - 1} X^{w_{[\ell]}}Z^1_{j_{\ell + 1}}W^{({_{[|w| - \ell - 1]}w})_{[i - \ell - 1]}}Z^2_{j_{i + 1}}Y^{_{[|w| - i - 1]}w}.
\end{aligned} \end{gather*}
Observing that 
$$_{[i-\ell-1]}(w_{[i]})=(_{[|w|-\ell-1]}w)_{[i-\ell-1]}=\left\{\begin{array}{ll}
g_{j_{\ell+2}}\cdots g_{j_i}, & 0\le\ell\le i-2,\\
\emptyset, & \ell=i-1,\end{array}\right.$$
we conclude that ${_0\Delta}p = {_1\Delta}p$, and by Theorem \ref{main} $p$ is integrable.

Conversely, consider a polynomial of the form $$p = \sum_{i = 1}^{|w|} p_{w, i}x^{w_{[i - 1]}}z_{j_i}y^{_{[|w| - i]}w}$$ and suppose that ${_0\Delta}p = {_1\Delta}p$.  To prove the result, we proceed by induction on the length of the word $w$. 

If $w$ has length $1$, then $p = p_{g_j}z_j$ is in the desired form: there is only one monomial. If $w$ has length $2$, then $p = p_{g_{j_1}g_{j_2}, 1}z_{j_1}y_{j_2} + p_{g_{j_1}g_{j_2}, 2}x_{j_1}z_{j_2}$. We compute  
 ${_0\Delta}p(X,W,Y)(Z^1,Z^2) =p_{g_{j_1}g_{j_2}, 2}Z^1_{j_1}Z^2_{j_2}$ and ${_1\Delta}p(X, W, Y)(Z^1, Z^2) = p_{g_{j_1}g_{j_2}, 1}Z^1_{j_1}Z^2_{j_2}$, so we conclude that $p_{g_{j_1}g_{j_2}, 1}=p_{g_{j_1}g_{j_2}, 2}$. 

Suppose by induction that $p_{w,i}$, $i=1,\ldots,|w|$, are all equal for any word $w$ of length $L - 1$. Then, for words of length $L$, we can write
\begin{align*}
\sum_{i = 1}^{L} p_{w, i}x^{w_{[i - 1]}}z_{j_i}y^{_{[L - i]}w} &= \(\sum_{i = 1}^{L - 1} p_{w, i}x^{w_{[i - 1]}}z_{j_i}y^{_{[L - 1-i]}(w_{[L - 1]})}\)y_{j_L} + p_{w, L}x^{w_{[L - 1]}}z_{j_L}. \end{align*}
Note that the term in parentheses is also a polynomial of the same form as $p$, but using the word $w_{[L - 1]}$ instead of $w$. Call this polynomial $\tilde{p}$. We have
\begin{align*}
{_0\Delta}p(X, W, Y)(Z^1, Z^2) &= {_0\Delta}\tilde{p}(X, W, Y)(Z^1, Z^2)Y_{j_L} \\
&\hspace{20mm} + \(p_{w, L}\sum_{i = 1}^{L - 2} X^{(w_{[L - 1]})_{[i - 1]}}Z^1_{j_i}W^{_{[L - i - 1]}(w_{[L - 1]})} \)Z^2_{j_L}, \\
{_1\Delta}p(X, W, Y)(Z^1, Z^2) &= {_1\Delta}\tilde{p}(X, W, Y)(Z^1, Z^2)Y_{j_L} + \tilde{p}(X, W)(Z^1)Z^2_{j_L}.
\end{align*}
Since we are assuming that ${_0\Delta}p = {_1\Delta}p$, the comparison of the coefficients gives us that 
$p_{w, i}=p_{w,L}$ for all $i=1,\ldots,L-1$, as needed.
\end{remark}

Proposition \ref{GHV} can be generalized to the case of an nc polynomial of arbitrary order. 

\begin{proposition}
An nc polynomial $p$ of order $k + 1$ is integrable with respect to ${_j\Delta}$, that is, there exists a nc polynomial $q$ of order $k$ such that ${_j\Delta}q=p$, if and only if for every monomial, \begin{multline*}
p_{(w,v)}(x^0)^{w_0}(z^1)^{v_1}(x^1)^{w_1} \cdots (z^{j})^{v_{j}}(x^j)^{({w_j})_{[i_0]}}z^{j + 1}_{\ell^{(j)}_{i_0 + 1}}(x^{j + 1})^{_{[|w_j| - i_0 - 1]}(w_j)}(z^{j + 2})^{v_{j + 1}}(x^{j + 2})^{w_{j + 1}}\\ \cdots (z^{k+1})^{v_{k}}(x^{k+1})^{w_{k}}\end{multline*} that is in $p$, 
where $w^j=\ell^{(j)}_1\cdots\ell^{(j)}_{|w|}$, the polynomial
\begin{multline*}
p_{(w,v)}\sum_{i = 0}^{|w_j| - 1}(x^0)^{w_0}(z^1)^{v_1}(x^1)^{w_1} \cdots (z^{j})^{v_{j}}(x^j)^{({w_j})_{[i_0]}}z^{j + 1}_{\ell^{(j)}_{i + 1}}(x^{j + 1})^{_{[|w_j| - i - 1]}(w_j)}(z^{j + 2})^{v_{j + 1}}(x^{j + 2})^{w_{j + 1}}\\ \cdots (z^{k+1})^{v_{k}}(x^{k+1})^{w_{k}}\end{multline*} is also in $p$. In this case, $q$ has the form
$$q=\sum_{\underset{|w| \leq L}{w \in\mathcal{G}}} \sum_{\underset{|v^1| = \hdots = |v^k| = 1}{v \in \mathcal{G}'}} p_{(w, v)} (x^0)^{w_0}(z^1)^{v_1}(x^1)^{w_1} \hdots (z^j)^{v_j}(x^j)^{w_j}(z^{j + 2})^{v_{j + 1}}(x^{j + 2})^{w_{j + 1}}\\ \cdots (z^{k+1})^{v_{k}}(x^{k+1})^{w_{k}}.$$
\end{proposition}

\subsection{The Functions are Complex Analytic}

In this section, we will consider three types of analyticity (see \cite[Chapter 7]{Verb}); in each case, it will be shown that when nc functions of order $k+1$, $F_0,\ldots,F_k$, have the given analyticity property, the antiderivative also has the same type of analyticity. 

Let $\V_j$ be vector spaces over $\mathbb{C}$, and let $\W_j$ be Banach spaces equipped with an admissible system of rectangular matrix norms on $\W_j^{p\times q}$, $p,q\in\mathbb{N}$, that is, all the rectangular block  injections
$$\iota^{(j; s_1,\ldots,s_m)}_{\alpha\beta}\colon\W_j^{s_\alpha\times s_\beta}\to\W_j^{s\times s},\quad W\mapsto E_{\alpha}WE_{\beta}^\top$$
and the projections
$$\pi^{(j;s_1,\ldots,s_m)}_{\alpha\beta}\colon\W_j^{s\times s}\to\W_j^{s_\alpha\times s_\beta},\quad W\mapsto E_{\alpha}^\top WE_\beta$$
are bounded linear operators, $j=0,\ldots,k$,  $s_1,\ldots,s_m\in\mathbb{N}$, $s=s_1+\cdots +s_m$, where%
$$E_{\alpha}=\text{col}[0_{s_1\times s_{\alpha}},\ldots,0_{s_{\alpha-1}\times s_{\alpha}},I_{s_{\alpha}},0_{s_{\alpha+1}\times s_{\alpha}},\ldots,0_{s_m\times s_{\alpha}} ].$$%

Let $\Omega^{(j)}\subseteq\V_{j, \nc}$ be finitely open nc sets, that is, for each $n\in\mathbb{N}$, the intersection of $\Omega^{(j)}_n$ with any finite-dimensional subspace $\U$ of $\V_j^{n\times n}$ is open in the Euclidean topology of $\U$. Recall \cite{Verb} that a nc function $f \in \T^k(\Omega^{(0)}, \ldots, \Omega^{(k)};\W_{0, \nc}, \ldots, \W_{k, \nc})$ is called

\begin{enumerate}
\item[I(a)] $W$-locally bounded on slices if for every $n_0, \ldots, n_k \in \NN$, $Y^j \in \Omega_{n_j}^{(j)}$, and $Z^j \in \V_j^{n_j \times n_j}$, $j = 0, \ldots, k$, and $W^j \in \W_j^{n_{j - 1} \times n_j}$, $j = 1, \ldots, k$, there exists a $\delta > 0$ such that
$$\sup_{t\in\mathbb{C}\colon |t|<\delta}\|f(Y^0 + tZ^0, \ldots, Y^k + tZ^k)(W^1, \ldots, W^k)\| < \infty.$$
\item[I(b)] $W$-Gateaux ($G_W$-) differentiable if for every $n_0, \ldots, n_k \in \NN$, $Y^j \in \Omega_{n_j}^{(j)}$, and $Z^j \in \V_j^{n_j \times n_j}$, $j = 0, \ldots, k$, and $W^j \in \W_j^{n_{j - 1} \times n_j}$ , $j = 1, \ldots, k$, the (complex) $G$-derivative,
$$\frac{d}{dt}f(Y^0 + tZ^0, \ldots, Y^k + tZ^k)(W^1, \ldots, W^k)\Big|_{t=0}$$
exists.
\item[I(c)] $W$-analytic on slices if for every $Y^j \in \Omega_{n_j}^{(j)}$ and $Z^j \in \V_j^{n_j \times n_j}$, $j = 0, \ldots, k$, and $W^j \in \W_j^{n_{j - 1} \times n_j}$, $j = 1, \ldots, k$, $f(Y^0 + tZ^0, \ldots, Y^k + tZ^k)(W^1, \ldots, W^k)$ is an analytic function of $t$ in a neighborhood of $0$.
\end{enumerate}

By Theorem 7.41 in \cite{Verb}, conditions I(a)--I(c) are all equivalent.
\begin{theorem}
Let $\Omega^{(0)}, \ldots, \Omega^{(k)}$ be right admissible nc sets. Suppose that $$F_j \in \T^{k + 1}(\Omega^{(0)}, \ldots,  \Omega^{(j-1)},\Omega^{(j)},\Omega^{(j)},\Omega^{(j+1)},\ldots,\Omega^{(k)};\W_{0, \nc}, \ldots, \W_{k, \nc})$$ are $W$-analytic on slices, $j = 0, \ldots, k$. Then there exists an nc function  $$f \in \T^k(\Omega^{(0)}, \ldots, \Omega^{(k)};\W_{0, \nc}, \ldots, \W_{k, \nc}),$$ which is $W$-analytic on slices and satisfies
$${_j\Delta}f = F_j,\quad j = 0, \ldots, k,$$
if and only if
$${_i\Delta}F_j = {_{j + 1}\Delta}F_i,\quad 0\le i\le j\le k.$$
Further, $f$ is uniquely determined up to a $k$-linear mapping as in Theorem \ref{higher-main} .
\end{theorem}

\begin{proof}
By Theorem \ref{higher-main}, the above conditions are equivalent to the existence of an antiderivative $f$. It suffices to show that $f$ is $W$-locally bounded on slices. We use \eqref{actual nicer form} in which $X^j = Y^j + tZ^j$ for $j = 0, \ldots k$:
\begin{multline*}
f(Y^0 + tZ^0, \ldots, Y^k + tZ^k)(W^1, \ldots, W^k) = f(Y^0, \ldots, Y^k)(W^1, \ldots, W^k) \\
+ \sum_{j = 0}^k F_j(Y^0, \ldots, Y^j, Y^j + tZ^j, \ldots, Y^k + tZ^k)(W^1, \ldots, W^j, tZ^j, W^{j + 1}, \ldots, W^k) \\
= f(Y^0, \ldots, Y^k)(W^1, \ldots, W^k)\\
 + t\sum_{j = 0}^k F_j(Y^0, \ldots, Y^j, Y^j + tZ^j, \ldots, Y^k + tZ^k)(W^1, \ldots, W^j, Z^j, W^{j + 1}, \ldots, W^k) .
\end{multline*}
Since the functions $F_j$ are $W$-locally bounded on slices, for fixed $Y^0, \ldots, Y^k$, $Z^0, \ldots, Z^k$, and $W^1, \ldots, W^k$, there exists a $\delta > 0$ such that
$$F_j(Y^0, \ldots, Y^j, Y^j + tZ^j, \ldots, Y^k + tZ^k)(W^1, \ldots, W^j, Z^j, W^{j + 1}, \ldots, W^k)$$
are bounded, say by $M_j>0$, for $|t| < \delta$.  Then
$$
\| f(Y^0 + tZ^0, \ldots, Y^k + tZ^k)(W^1, \ldots, W^k)\|\le \| f(Y^0, \ldots, Y^k)(W^1, \ldots, W^k)\| +\delta\sum_{j=0}^kM_j.
$$
It follows that $f$ is $W$-locally bounded on slices, and hence, $W$-analytic on slices as desired.
\end{proof}

For our second notion of analyticity, we require additionally that $\V_j$ are complex Banach spaces equipped with an admissible system of rectangular matrix norms and that $\Omega^{(j)}\subseteq\V_{j,\nc}$ are open nc sets, i.e., $\Omega^{(j)}_n$ are open in $\V_j^{n\times n}$, $n=1,2,\ldots$, $j=0,\ldots,k$. We
say that an nc function $$f \in \T^k(\Omega^{(0)}, \ldots, \Omega^{(k)};\W_{0, \nc}, \ldots, \W_{k, \nc})$$ is called

\begin{enumerate}
\item[II(a)] locally bounded if for every $n_0, \ldots, n_k \in \NN$ and $Y^j \in \Omega^{(j)}$, $j = 0, \ldots k$, there exists a $\delta_j > 0$ such that 
$$\sup_{\|X^j-Y^j\|<\delta_j,\, j = 0, \ldots, k}\|f(X^0, \ldots, X^k)\|_{\L^k} < \infty,$$
where $\|\cdot\|_{\L^k}$ is the norm of a $k$-linear form:
$$\|\omega\|_{\L^k}=\sup_{\|W^1\|=\ldots=\|W^k\|=1}\|\omega(W^1,\ldots,W^k)\|.$$

\item[II(b)] Gateaux ($G$)-differentiable if for every $n_0, \ldots, n_k \in \NN$, $Y^j \in \Omega_{n_j}^{(j)}$, and $Z^j \in \V_j^{n_j \times n_j}$, $j = 0, \ldots k$,
$$\frac{d}{dt}f(Y^0 + tZ^0, \ldots, Y^k + tZ^k)\Big|_{t=0}$$
exists in the norm $\|\cdot\|_{\L^k}$.

\item[II(c)] analytic if $f$ is locally bounded and $G$-differentiable.
\end{enumerate}

By Theorem 7.46 in \cite{Verb}, II(a)--II(c) are all equivalent. 

\begin{theorem} \label{analytic}
Let $\Omega^{(0)}, \ldots, \Omega^{(k)}$ be open nc sets. Suppose that $$F_j \in \T^{k + 1}(\Omega^{(0)}, \ldots, \Omega^{(j-1)},\Omega^{(j)},\Omega^{(j)},\Omega^{(j+1)},\ldots,\Omega^{(k)};\W_{0, \nc}, \ldots, \W_{k, \nc})$$ are analytic for $j = 0, \ldots, k$. Then there exists an analytic nc function $$f \in \T^k(\Omega^{(0)}, \ldots, \Omega^{(k)};\W_{0, \nc}, \ldots, \W_{k, \nc})$$ such that $${_j\Delta}f = F_j, \quad j = 0, \ldots, k,$$
if and only if
$${_i\Delta}F_j = {_{j + 1}\Delta}F_i, \quad 0\le i\le j\le k.$$
Further, $f$ is uniquely determined up to a $k$-linear mapping as in Theorem \ref{higher-main}. 
\end{theorem}

\begin{proof}
We again use  \eqref{actual nicer form} so that $f$ is given by
\begin{multline*}
f(X^0, \ldots, X^k)(W^1, \ldots, W^k) = f(Y^0, \ldots, Y^k)(W^1, \ldots, W^k) \\
+ \sum_{j = 0}^k F_j(Y^0, \ldots, Y^j, X^j, \ldots, X^k)(W^1, \ldots, W^j, X^j - Y^j, W^{j + 1}, \ldots, W^k). \end{multline*}
Since the nc functions $F_0, \ldots, F_k$ are analytic, they are locally bounded. Thus, for every $Y^j \in \Omega^{(j)}$, there exists a $\delta_j$ such that $$\|F_j(X^0,\ldots, X^k)\|_{\L^k}$$ is bounded, say by $M_j$ when $\|X^j-Y^j\|< \delta_j$, $j= 0, \ldots, k$. It follows from Theorem \ref{derivation properties}, Theorem \ref{almost result again}, Proposition \ref{with zero functions}, Proposition \ref{g}, and Corollary \ref{nicer form} that $f(Y^0,\ldots, Y^k)$ can be chosen a bounded $k$-linear mapping. Then
$$\|f(X^0,\ldots,X^k)\|_{\L^k}\le\|f(Y^0,\ldots,Y^k)\|_{\L^k}+\sum_{j=0}^k\delta_jM_j$$ when $\|X^j-Y^j\|< \delta_j$, $j = 0, \ldots, k$. It follows that $f$ is locally bounded and thus analytic.
\end{proof}

Finally, for our third notion of analyticity, we assume that $\V_j$, $\W_j$ are operator spaces and $\Omega^{(j)}$ are uniformly open nc sets, $j=0,\ldots,k$, that is, for every $s\in\mathbb{N}$ and $Y\in\Omega^{(j)}_s$, there exists a nc ball
$$B_{\nc}(Y,\epsilon)=\coprod_{m=1}^\infty\{X\in\V_j^{sm\times sm}\colon \|X-I_m\otimes Y\|<\epsilon\}$$
that is contained in $\Omega^{(j)}$. We say that an nc function $$f \in \T^k(\Omega^{(0)}, \ldots, \Omega^{(k)};\W_{0, \nc}, \ldots, \W_{k, \nc})$$ is called

\begin{enumerate}
\item[III(a)] uniformly locally completely bounded if for every $Y^j \in \Omega^{(j)}$, $j = 0, \ldots, k$, there exists a $\delta_j > 0$ such that 
$$\sup_{X^j \in B_{\nc}(Y^j, \delta_j),\, j=0,\ldots,k}\|f(X^0, \ldots, X^k)\|_{\L^k_{\cb}} < \infty.$$
 Here, for a $k$-linear mapping $\omega\colon \W_1\times \cdots\times\W_k\to\W_0$, we define
$$\|\omega\|_{\L^k_{\cb}}=\sup_{n_0,\ldots,n_k\in\mathbb{N}}\|\omega^{(n_0,\ldots,n_k)}\|_{\L^k}$$
where the $k$-linear mapping $\omega^{(n_0,\ldots,n_k)}\colon\W_1^{n_0\times n_1}\times\cdots\times\W^{n_{k-1}\times n_k}\to\W_0^{n_0\times n_k}$ is defined by
$$\omega(W^1,\ldots,W^k)=(W^1\odot\cdots\odot W^k)\omega,$$
$W^1\odot\cdots\odot W^k\in(\W_1\otimes\cdots\otimes\W_k)^{n_0\times n_k}$ is the usual product of matrices, however the entries are multiplied using tensor products, and $\omega$ acts on this matrix on the right entrywise -- we identify $k$-linear forms and associated linear forms on a tensor product of $k$ factors; see details in Section 7.4 of \cite{Verb}.

\item[III(b)] completely bounded Gateaux ($G_{\cb}$)-differentiable if for every $n_0, \ldots, n_k \in \NN$, $Y^j \in \Omega_{n_j}^{(j)}$, and $Z^j \in \V_j^{n_j \times n_j}$, $j = 0, \ldots k$,
$$\frac{d}{dt}f(Y^0 + tZ^0, \ldots, Y^k + tZ^k)\Big|_{t=0}$$
exists in the norm $\|\cdot\|_{\L^k_{\cb}}$.
\item[III(c)] uniformly completely bounded (uniformly cb-) analytic if $f$ is uniformly locally \\*completely bounded and $G_{\cb}$-differentiable.
\end{enumerate}

By Proposition 7.53 in \cite{Verb}, III(a)--III(c) are all equivalent.

\begin{theorem}
Let $\Omega^{(0)}, \ldots, \Omega^{(k)}$ be uniformly open nc sets. Suppose that $$F_j \in \T^{k + 1}(\Omega^{(0)}, \ldots, \Omega^{(j-1)},\Omega^{(j)},\Omega^{(j)},\Omega^{(j+1)},\ldots,\Omega^{(k)};\W_{0, \nc}, \ldots, \W_{k, \nc})$$ are uniformly cb-analytic for $j = 0, \ldots, k$. Then there exists a uniformly cb-analytic nc function $$f \in \T^k(\Omega^{(0)}, \ldots, \Omega^{(k)};\N_{0, \nc}, \ldots, \N_{k, \nc})$$ such that
$${_j\Delta}f = F_j \hspace{10mm} j = 0, \ldots, k,$$
if and only if
$${_i\Delta}F_j = {_{j + 1}\Delta}F_i, \hspace{10mm} j = 0, \ldots, k.$$
Further, $f$ is uniquely determined up to a $k$-linear mapping as in \ref{higher-main}. 
\end{theorem}

\begin{proof}
We again use  \eqref{actual nicer form} so that $f$ is given by
\begin{multline*}
f(X^0, \ldots, X^k)(W^1, \ldots, W^k)  
= f(I_{m_0} \otimes Y^0, \ldots, I_{m_k} \otimes Y^k)(W^1, \ldots, W^k) \\
+ \sum_{j = 0}^k F_j(I_{m_0} \otimes Y^0, \ldots, I_{m_j} \otimes Y^j, X^j, \ldots, X^k)(W^1, \ldots, W^j, X^j - I_{m_j} \otimes Y^j, W^{j + 1}, \ldots, W^k) .\end{multline*}
Since the nc functions $F_0, \ldots, F_k$ are uniformly cb-analytic, they are uniformly locally completely bounded. Thus, for every $Y^j \in \Omega^{(j)}$, there exists a $\delta_j$ such that $$\|F_j( X^0,  \ldots, X^k)\|_{\L^k_{\cb}}$$ is bounded on $B_{\nc}(Y^0,\delta_0)\times\cdots\times B_{\nc}(Y^k,\delta_k)$, say by $M_j$, $j= 0, \ldots, k$. It follows from Theorem \ref{derivation properties}, Theorem \ref{almost result again}, Proposition \ref{with zero functions}, Proposition \ref{g}, and Corollary \ref{nicer form} that $f(Y^0,\ldots, Y^k)$ can be chosen a completely bounded $k$-linear mapping. Then
$$\|f(X^0,\ldots,X^k)\|_{\L^k_{\cb}}\le\|f(Y^0,\ldots,Y^k)\|_{\L^k_{\cb}}+\sum_{j=0}^k\delta_jM_j$$ when $X^j\in B_{\nc}(Y^j,\delta_j)$, $j = 0, \ldots, k$. It follows that $f$ is uniformly locally completely bounded and thus uniformly cb-analytic.\end{proof}

\end{document}